\documentclass{article}

\usepackage{arxiv}

\usepackage[utf8]{inputenc}
\usepackage[T1]{fontenc}
\usepackage{amsmath}
\usepackage{amsfonts}
\usepackage{amssymb}
\usepackage{amsthm}
\usepackage{mathtools}
\usepackage{tikz}
\usepackage{pgfplots}
\usepackage{algorithm}%
\usepackage{algorithmicx}%
\usepackage[noend]{algpseudocode}
\usepackage[colorlinks=true, linkcolor=blue]{hyperref}
\usepackage{cleveref}
\usepackage{booktabs}
\usepackage{mdframed}
\usepackage{appendix}

\usepgfplotslibrary{external}
\tikzset{external/system call = {%
    pdflatex \tikzexternalcheckshellescape%
    -halt-on-error
    -interaction=batchmode
    -jobname "\image" "\texsource"}}
\newtheorem{theorem}{Theorem}[section]

\newtheorem{lemma}[theorem]{Lemma}
\newtheorem{corollary}[theorem]{Corollary}

\newtheorem{remark}[theorem]{Remark}

\newtheorem{assumption}[theorem]{Assumption}

\algrenewcommand\algorithmicrequire{\textbf{Input:}}
\algrenewcommand\algorithmicensure{\textbf{Output:}}

\numberwithin{equation}{section}

\usepackage{amsmath}
\usepackage{amsfonts}
\usepackage{amssymb}
\usepackage{mathtools}
\usepackage{tikz}
\usepackage{pgfplots}
\usepackage{dsfont}

\usepackage{acronym}
\acrodef{fom}[FOM]{full-order model}
\acrodef{lti}[LTI]{linear time-invariant}
\acrodef{mse}[MSE]{mean square error}
\acrodef{pde}[PDE]{partial differential equation}
\acrodef{pod}[POD]{proper orthogonal decomposition}
\acrodef{rb}[RB]{reduced basis method}
\acrodef{rom}[ROM]{reduced-order model}
\acrodef{strom}[StROM]{structured reduced-order model}
\acrodef{mor}[MOR]{model order reduction}
\acrodef{ls}[LS]{least-squares}
\acrodefindefinite{lti}{an}{a}
\acrodefindefinite{mse}{an}{a}

\newcommand{\nfom}{\ensuremath{n}}
\newcommand{\nin}{\ensuremath{n_{\textnormal{i}}}}
\newcommand{\nout}{\ensuremath{n_{\textnormal{o}}}}

\newcommand{\Af}{\ensuremath{\mathcal{A}}}
\newcommand{\Bf}{\ensuremath{\mathcal{B}}}
\newcommand{\Cf}{\ensuremath{\mathcal{C}}}
\newcommand{\Ef}{\ensuremath{\mathcal{E}}}

\newcommand{\cAf}{\ensuremath{A}}
\newcommand{\cBf}{\ensuremath{B}}
\newcommand{\cCf}{\ensuremath{C}}

\newcommand{\xf}{\ensuremath{x}}
\newcommand{\yf}{\ensuremath{y}}

\newcommand{\nrom}{\ensuremath{r}}

\newcommand{\Ar}{\ensuremath{\hat{\mathcal{A}}}}
\newcommand{\Br}{\ensuremath{\hat{\mathcal{B}}}}
\newcommand{\Cr}{\ensuremath{\hat{\mathcal{C}}}}
\newcommand{\Er}{\ensuremath{\hat{\mathcal{E}}}}

\newcommand{\cAr}{\ensuremath{\hat{A}}}
\newcommand{\cBr}{\ensuremath{\hat{B}}}
\newcommand{\cCr}{\ensuremath{\hat{C}}}

\newcommand{\car}{\ensuremath{\hat{\alpha}}}
\newcommand{\cbr}{\ensuremath{\hat{\beta}}}
\newcommand{\ccr}{\ensuremath{\hat{\gamma}}}

\newcommand{\qAr}{\ensuremath{q_{\Ar}}}
\newcommand{\qBr}{\ensuremath{q_{\Br}}}
\newcommand{\qCr}{\ensuremath{q_{\Cr}}}

\newcommand{\xr}{\ensuremath{\hat{x}}}
\newcommand{\yr}{\ensuremath{\hat{y}}}

\newcommand{\xrd}{\ensuremath{\xr_{d}}}

\newcommand{\hA}{\ensuremath{\hat{A}}}
\newcommand{\hB}{\ensuremath{\hat{B}}}
\newcommand{\hC}{\ensuremath{\hat{C}}}
\newcommand{\hE}{\ensuremath{\hat{E}}}
\newcommand{\hG}{\ensuremath{\hat{G}}}
\newcommand{\hH}{\ensuremath{\hat{H}}}
\newcommand{\hS}{\ensuremath{\hat{S}}}
\newcommand{\hT}{\ensuremath{\hat{T}}}
\newcommand{\hX}{\ensuremath{\hat{X}}}
\newcommand{\hY}{\ensuremath{\hat{Y}}}

\newcommand{\obj}{\ensuremath{\mathcal{J}}}

\newcommand{\cH}{\ensuremath{\mathcal{H}}}
\newcommand{\cL}{\ensuremath{\mathcal{L}}}
\newcommand{\Htwo}{\ensuremath{\cH_{2}}}
\newcommand{\htwo}{\ensuremath{h_{2}}}
\newcommand{\Ltwo}{\ensuremath{\cL_{2}}}

\newcommand{\Linf}{\ensuremath{\cL_{\infty}}}

\newcommand{\HtwoLtwo}{\ensuremath{\cH_{2} \otimes \cL_{2}}}

\newcommand{\pp}{\ensuremath{\mathsf{p}}}
\newcommand{\pset}{\ensuremath{\mathcal{P}}}
\newcommand{\npar}{\ensuremath{n_{\textnormal{p}}}}
\newcommand{\ppp}{\ensuremath{\xi}}
\newcommand{\psetp}{\ensuremath{\Xi}}

\newcommand{\CC}{\ensuremath{\mathbb{C}}}
\newcommand{\DD}{\ensuremath{\mathbb{D}}}
\newcommand{\RR}{\ensuremath{\mathbb{R}}}

\newcommand{\CCpar}{\ensuremath{\CC^{\npar}}}

\newcommand{\CCi}{\ensuremath{\CC^{\nin}}}
\newcommand{\CCo}{\ensuremath{\CC^{\nout}}}
\newcommand{\CCr}{\ensuremath{\CC^{\nrom}}}

\newcommand{\CCoi}{\ensuremath{\CC^{\nout \times \nin}}}

\newcommand{\CCor}{\ensuremath{\CC^{\nout \times \nrom}}}
\newcommand{\CCri}{\ensuremath{\CC^{\nrom \times \nin}}}
\newcommand{\CCro}{\ensuremath{\CC^{\nrom \times \nout}}}
\newcommand{\CCrr}{\ensuremath{\CC^{\nrom \times \nrom}}}

\newcommand{\RRf}{\ensuremath{\RR^{\nfom}}}
\newcommand{\RRi}{\ensuremath{\RR^{\nin}}}
\newcommand{\RRo}{\ensuremath{\RR^{\nout}}}
\newcommand{\RRr}{\ensuremath{\RR^{\nrom}}}
\newcommand{\RRff}{\ensuremath{\RR^{\nfom \times \nfom}}}
\newcommand{\RRfi}{\ensuremath{\RR^{\nfom \times \nin}}}
\newcommand{\RRof}{\ensuremath{\RR^{\nout \times \nfom}}}
\newcommand{\RRoi}{\ensuremath{\RR^{\nout \times \nin}}}

\newcommand{\RRor}{\ensuremath{\RR^{\nout \times \nrom}}}
\newcommand{\RRri}{\ensuremath{\RR^{\nrom \times \nin}}}

\newcommand{\RRrr}{\ensuremath{\RR^{\nrom \times \nrom}}}

\DeclarePairedDelimiter{\myparen}{\lparen}{\rparen}
\DeclarePairedDelimiter{\mybrack}{\lbrack}{\rbrack}
\DeclarePairedDelimiter{\mybrace}{\lbrace}{\rbrace}

\DeclarePairedDelimiter{\abs}{\lvert}{\rvert}

\DeclarePairedDelimiterXPP{\mydiag}[1]{\operatorname{diag}}{\lparen}{\rparen}{}{#1}
\DeclarePairedDelimiterXPP{\normtwo}[1]{}{\lVert}{\rVert}{_{2}}{#1}
\DeclarePairedDelimiterXPP{\normF}[1]{}{\lVert}{\rVert}{_{\operatorname{F}}}{#1}
\DeclarePairedDelimiterXPP{\normHtwo}[1]{}{\lVert}{\rVert}{_{\Htwo}}{#1}
\DeclarePairedDelimiterXPP{\normhtwo}[1]{}{\lVert}{\rVert}{_{\htwo}}{#1}
\DeclarePairedDelimiterXPP{\normLtwo}[1]{}{\lVert}{\rVert}{_{\Ltwo}}{#1}
\DeclarePairedDelimiterXPP{\normLtwomu}[1]{}{\lVert}{\rVert}{_{\Ltwo(\pset, \measure)}}{#1}
\DeclarePairedDelimiterXPP{\normHtwoLtwo}[1]{}{\lVert}{\rVert}{_{\HtwoLtwo}}{#1}
\DeclarePairedDelimiterXPP{\normLinf}[1]{}{\lVert}{\rVert}{_{\Linf}}{#1}
\DeclarePairedDelimiterXPP{\normLinfmu}[1]{}{\lVert}{\rVert}{_{\Linf(\pset, \measure)}}{#1}
\DeclarePairedDelimiterXPP{\ip}[2]{}{\langle}{\rangle}{}{#1, #2}
\DeclarePairedDelimiterXPP{\ipF}[2]{}{\langle}{\rangle}{_{\operatorname{F}}}{#1, #2}
\DeclarePairedDelimiterXPP{\ipHtwo}[2]{}{\langle}{\rangle}{_{\Htwo}}{#1, #2}
\DeclarePairedDelimiterXPP{\iphtwo}[2]{}{\langle}{\rangle}{_{\htwo}}{#1, #2}
\DeclarePairedDelimiterXPP{\ipLtwo}[2]{}{\langle}{\rangle}{_{\Ltwo}}{#1, #2}
\DeclarePairedDelimiterXPP{\ipLtwomu}[2]{}{\langle}{\rangle}{_{\Ltwo(\pset, \measure)}}{#1, #2}
\DeclarePairedDelimiterXPP{\trace}[1]{\operatorname{tr}}{\lparen}{\rparen}{}{#1}
\DeclarePairedDelimiterXPP{\vecop}[1]{\operatorname{vec}}{\lparen}{\rparen}{}{#1}
\DeclarePairedDelimiterXPP{\Real}[1]{\operatorname{Re}}{\lparen}{\rparen}{}{#1}
\DeclarePairedDelimiterXPP{\myspan}[1]{\operatorname{span}}{\lbrace}{\rbrace}{}{#1}

\newcommand{\measure}{\ensuremath{\mu}}

\DeclareMathOperator{\dif}{d\!}

\DeclareMathOperator*{\argmin}{arg\,min}

\DeclareMathOperator*{\esssup}{ess\,sup}

\newcommand{\difm}[1]{\ensuremath{\dif{\measure(#1)}}}

\newcommand{\fundef}[3]{\ensuremath{#1 \colon #2 \to #3}}

\newcommand{\tran}{^{\operatorname{T}}}
\newcommand{\mtran}{^{-\!\operatorname{T}}}
\newcommand{\herm}{^{*}}
\newcommand{\mherm}{^{-*}}

\newcommand{\imag}{\boldsymbol{\imath}}

\newcommand{\ones}{\ensuremath{\mathds{1}}}

\let\hat\widehat%
\definecolor{mplC0}{HTML}{1F77B4}
\definecolor{mplC1}{HTML}{FF7F0E}
\definecolor{mplC2}{HTML}{2CA02C}
\definecolor{mplC3}{HTML}{D62728}
\definecolor{mplC4}{HTML}{9467BD}
\definecolor{mplC5}{HTML}{8C564B}
\definecolor{mplC6}{HTML}{E377C2}
\definecolor{mplC7}{HTML}{7F7F7F}
\definecolor{mplC8}{HTML}{BCBD22}
\definecolor{mplC9}{HTML}{17BECF}

\pgfplotsset{compat=1.10}
\pgfplotscreateplotcyclelist{mpl}{
  mplC0, very thick \\
  mplC1, very thick, densely dashed \\
  mplC2, very thick, densely dotted \\
  mplC3, very thick, densely dashdotted \\
}
\pgfplotscreateplotcyclelist{mplmark}{
  mplC0, very thick, mark=o, every mark/.append style={solid} \\
  mplC1, very thick, densely dashed, mark=x, every mark/.append style={solid}, mark size=3 \\
  mplC2, very thick, densely dotted, mark=square, every mark/.append style={solid} \\
  mplC3, very thick, densely dashdotted, mark=diamond, every mark/.append style={solid} \\
}
\pgfplotscreateplotcyclelist{mplmarkonly}{
  draw=mplC0, fill=mplC0, only marks, mark=*, mark size=1 \\
  draw=mplC1, fill=mplC1, only marks, mark=square*, mark size=1 \\
  draw=mplC2, fill=mplC2, only marks, mark=triangle*, mark size=1.5 \\
}

\begin{document}

\title{A Unifying Framework for Interpolatory
  \texorpdfstring{$\Ltwo$}{L2}-optimal Reduced-order Modeling\thanks{%
    This work was partially funded by the U.S. National Science Foundation under
    grant DMS-1923221.
    Parts of this material are based upon work supported by the National
    Science Foundation under Grant No.\ DMS-1929284 while the authors were in
    residence at the Institute for Computational and Experimental Research in
    Mathematics in Providence, RI, during the Spring 2020 Reunion Event for
    Model and Dimension Reduction in Uncertain and Dynamic Systems program.
  }}

\author{%
  \href{https://orcid.org/0000-0002-9437-7698}{%
    \includegraphics[scale=0.06]{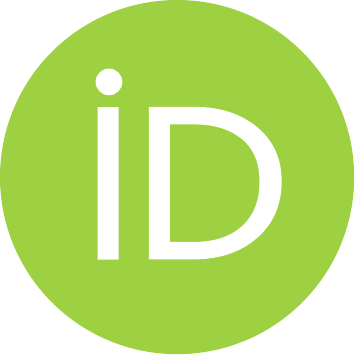}}\hspace{1mm}Petar~Mlinari\'c%
  \thanks{Department of Mathematics, Virginia Tech, Blacksburg, VA 24061
    (\texttt{mlinaric@vt.edu}).}
	\And
	\href{https://orcid.org/0000-0003-4564-5999}{%
    \includegraphics[scale=0.06]{orcid.pdf}}\hspace{1mm}Serkan~Gugercin%
  \thanks{Department of Mathematics and Division of Computational Modeling and
    Data Analytics, Academy of Data Science, Virginia Tech, Blacksburg, VA 24061
    (\texttt{gugercin@vt.edu}).}
}

\renewcommand{\headeright}{Technical Report}
\renewcommand{\undertitle}{Technical Report}
\renewcommand{\shorttitle}{A Unifying Framework for Interpolatory
  $\Ltwo$-optimal Modeling}

\hypersetup{
  pdftitle={A Unifying Framework for Interpolatory L2-optimal Reduced-order
    Modeling},
  pdfauthor={Petar~Mlinari\'c, Serkan~Gugercin},
}

\maketitle

\begin{abstract}
  We develop a unifying framework
  for interpolatory $\Ltwo$-optimal reduced-order modeling
  for a wide classes of problems
  ranging from stationary models to parametric dynamical systems.
  We first show that
  the framework naturally covers the well-known interpolatory necessary
  conditions for $\Htwo$-optimal model order reduction and
  leads to the interpolatory conditions for $\HtwoLtwo$-optimal model order
  reduction of multi-input/multi-output parametric dynamical systems.
  Moreover,
  we derive novel interpolatory optimality conditions
  for rational discrete least-squares minimization  and
  for $\Ltwo$-optimal model order reduction of a class of parametric stationary
  models.
  We show that bitangential Hermite interpolation appears as the main tool
  for optimality across different domains.
  The theoretical results are illustrated on two numerical examples.
\end{abstract}


\section{Introduction}
Interpolatory methods have been one of the most commonly used \ac{mor}
techniques, see, e.g.,~\cite{AntBG20,Benetal17,BeaG17}.
For $\Htwo$-optimal \ac{mor} of \ac{lti} dynamical systems,
the necessary optimality conditions are known and appear in the form of
(bitangential) Hermite interpolation of the underlying transfer
function~\cite{MeiL67,GugAB06,GugAB08,AntBG10}.
These interpolatory optimality conditions have formed the foundation of various
algorithms and have been extended to different settings; see,
e.g.,~\cite{GugPBS12,GugSW13,FlaBG13,AniBGA13,FlaG15,BreBG15,BenGG18,BenB12}.
But in various other important settings,
such as in the optimal approximation of stationary problems and
discrete \ac{ls} rational fitting,
it is not yet established whether the optimality requires interpolation
(as in the $\Htwo$-case) and if so,
what those interpolation conditions are.
For instance,
reduced basis methods use a greedy selection of sampling (interpolation) points
to match the solution at these points~\cite{HesRS16,QuaMN16}.
Is there an underlying framework for interpolatory optimality conditions?

The authors recently developed a data-driven framework for $\Ltwo$-optimal
reduced-order modeling of parametric systems~\cite{MliG22}.
In this paper, we show how~\cite{MliG22} provides a unifying framework for
interpolatory optimal approximation both for dynamical systems and stationary
problems.
We prove that
bitangential Hermite interpolation is the necessary condition for optimality
not only for approximation of \ac{lti} systems in the $\Htwo$ norm,
but also in many other prominent cases,
thus extending the optimal interpolation theory to a broader class of
problems.

First we recall the $\Ltwo$-optimal reduced-order modeling problem discussed
in~\cite{MliG22}:
Consider a parameter-to-output mapping
\begin{align}\label{eq:ymapping}
  \fundef{\yf}{\pset}{\CCoi}
\end{align}
where $\pset \subseteq \CCpar$ and $\nin, \nout, \npar$ are positive integers.
Assume that evaluating $\yf(\pp)$ is expensive.
Thus the goal is to construct a high-fidelity reduced-order mapping
(approximation) $\fundef{\yr}{\pset}{\CCoi}$,
which is much cheaper to evaluate than $\yf$.
Inspired by the structures arising in projection-based (parametric) \ac{mor},
\cite{MliG22} constructs a structured \ac{rom}
\begin{subequations}\label{eq:rom}
  \begin{align}
    \Ar(\pp) \xr(\pp) & = \Br(\pp), \\*
    \yr(\pp) & = \Cr(\pp) \xr(\pp),
  \end{align}
\end{subequations}
with a parameter-separable form
\begin{equation}\label{eq:rom-param-sep-form}
  \Ar(\pp) = \sum_{i = 1}^{\qAr} \car_i(\pp) \cAr_i, \quad
  \Br(\pp) = \sum_{j = 1}^{\qBr} \cbr_j(\pp) \cBr_j, \quad
  \Cr(\pp) = \sum_{k = 1}^{\qCr} \ccr_k(\pp) \cCr_k,
\end{equation}
where
$\xr(\pp) \in \CCr$ is the reduced state,
$\yr(\pp) \in \CCoi$ is the approximate output,
$\Ar(\pp) \in \CCrr$,
$\Br(\pp) \in \CCri$,
$\Cr(\pp) \in \CCor$,
$\fundef{\car_i, \cbr_j, \ccr_k}{\pset}{\CC}$,
$\cAr_i \in \RRrr$,
$\cBr_j \in \RRri$, and
$\cCr_k \in \RRor$.
Note that when the reduced-order dimension $\nrom$ is small,
evaluating the structured \ac{rom}, i.e., evaluating $\yr(\pp)$, is cheap.
In~\cite{MliG22}, we showed that the structure of the \ac{rom}
in~\eqref{eq:rom}--\eqref{eq:rom-param-sep-form} covers a wide range of problems
including (parametric) \ac{lti} systems and
models arising from discretization of stationary parametric partial differential
equations.
We revisit some concrete choices of $\car_i, \cbr_j, \ccr_k$ later in the paper.
In~\cite{MliG22}, we developed numerical algorithms to construct the
\ac{rom}~\eqref{eq:rom} in a purely data-driven fashion and called it a
data-driven \ac{rom}.
In this paper, we simply call it \iac{strom}.

In order to judge the quality of a \ac{rom}, one needs an error measure.
In~\cite{MliG22}, we constructed the \ac{strom} to minimize the squared $\Ltwo$
error
\begin{equation}\label{eq:l2-cost}
  \obj\myparen*{\cAr_i, \cBr_j, \cCr_k}
  = \normLtwomu*{\yf - \yr}^2
  = \int_{\pset} \normF*{\yf(\pp) - \yr(\pp)}^2 \difm{\pp}
\end{equation}
and derived the gradients of $\obj$ with respect to the \ac{strom} matrices
$\cAr_i, \cBr_j, \cCr_k$.
These gradient formulae, which we also recall in \Cref{sec:prelim},
were then used in developing an optimization-based reduced-order modeling
algorithm.

Starting with the formulation of~\cite{MliG22},
our goals here are to develop
a unifying framework for interpolatory $\Ltwo$-optimal reduced-order modeling
that covers both stationary and dynamical problems and
to prove that bitangential Hermite interpolation
is the necessary conditions for optimality
in a much broader classes of problems than previously studied.
More specifically, our main contributions are as follows:
\begin{enumerate}
\item We show that the existing interpolatory optimality conditions for
  approximating \ac{lti} systems is a special case of our formulation and
  directly follows from it (\Cref{{sec:lti-h2}}).
\item We derive interpolatory optimality conditions for approximating parametric
  \ac{lti} systems (\Cref{sec:grimm}).
\item We derive interpolatory optimality conditions for rational discrete
  \ac{ls} measure (\Cref{sec:lti-ls}).
\item We derive interpolatory optimality conditions for approximation of
  parametric stationary problems (\Cref{sec:stat}).
\end{enumerate}

The rest of paper is organized as follows.
In \Cref{sec:prelim},
we recall some of the main results from~\cite{MliG22} and
give the necessary optimality conditions,
which we use repeatedly throughout the paper.
In \Cref{sec:lti},
we show applications to \ac{lti} systems, both parametric and non-parametric,
using the continuous, Lebesgue measure.
We consider the discrete \ac{ls} measure in \Cref{sec:lti-ls},
where we derive interpolatory conditions for the \ac{ls} problem.
In \Cref{sec:stat},
we consider a class of stationary parametric systems and
derive interpolatory conditions for the $\Ltwo$-optimal \acp{rom}.
Conclusions are given in \Cref{sec:conclusion}.

\section{Mathematical Preliminaries}%
\label{sec:prelim}
Here we recall one of the main results of~\cite{MliG22},
namely the gradients of $\obj$~\eqref{eq:l2-cost} with respect to the \ac{strom}
matrices, and
then present the necessary optimality conditions that immediately
follow from this result.

\subsection{Gradients of the Squared \texorpdfstring{$\Ltwo$}{L2} Error}
We begin with the necessary assumptions.
\begin{assumption}\label{assumption:all}
  For the problem setup in~\eqref{eq:ymapping}--\eqref{eq:l2-cost},
  let the following hold:
  \begin{enumerate}
  \item The set $\pset \subseteq \CCpar$ is closed under conjugation
    ($\overline{\pp} \in \pset$ for all $\pp \in \pset$).
  \item The measure $\measure$ over $\pset$ is closed under conjugation
    (for any measurable set $S$,
    $\overline{S}$ is measurable and
    $\measure(\overline{S}) = \measure(S)$).
  \item The function $\fundef{\yf}{\pset}{\CCoi}$ is
    measurable,
    closed under conjugation
    ($\overline{\yf(\pp)} = \yf(\overline{\pp})$ for all $\pp \in \pset$), and
    square-integrable
    ($\normLtwomu{\yf} < \infty$).
  \item The scalar functions $\fundef{\car_i, \cbr_j, \ccr_k}{\pset}{\CC}$, for
    $i = 1, 2, \ldots, \qAr$,
    $j = 1, 2, \ldots, \qBr$, and
    $k = 1, 2, \ldots, \qCr$,
    are
    measurable,
    closed under conjugation, and
    \begin{equation}\label{eq:abc-l2-bounded}
      \int_{\pset}
      \myparen*{
        \frac{
          \sum_{j = 1}^{\qBr} \abs*{\cbr_j(\pp)}
          \sum_{k = 1}^{\qCr} \abs*{\ccr_k(\pp)}
        }{
          \sum_{i = 1}^{\qAr} \abs{\car_i(\pp)}
        }
      }^{\!\!2}
      \difm{\pp}
      < \infty.
    \end{equation}
  \item\label{item:inv}
    The matrices $\cAr_1, \cAr_2, \ldots, \cAr_{\qAr}$ are such that
    \begin{equation*}
      \esssup_{\pp \in \pset} \, \normF*{\car_i(\pp) {\Ar(\pp)}^{-1}}
      < \infty, \quad
      i = 1, 2, \ldots, \qAr,
    \end{equation*}
    where $\Ar$ is as in~\eqref{eq:rom-param-sep-form}.
  \end{enumerate}
\end{assumption}
These assumptions trivially hold in many cases, see~\cite{MliG22}.
Note that item~\ref{item:inv} requires $\Ar(\pp)$ be invertible for
$\measure$-almost all $\pp$ in $\pset$.
\begin{theorem}[Theorem~3.7 in~\cite{MliG22}]\label{thm:grad}
  Let $\pset$, $\measure$, $\yf$, $\car_i$, $\cbr_j$, $\ccr_k$, and $\cAr_i$
  satisfy \Cref{assumption:all}.
  Then, the gradients of $\obj$ with respect to the \ac{strom} matrices are
  \begin{subequations}\label{eq:grad}
    \begin{align}
      \label{eq:grad-A}
      \nabla_{\cAr_i} \obj = {}
      &
        2 \int_{\pset} \car_i(\overline{\pp})
        \xrd(\pp)
        \mybrack*{\yf(\pp) - \yr(\pp)}
        \xr(\pp)\herm
        \difm{\pp},
      & i = 1, 2, \ldots, \qAr, \\*
      \label{eq:grad-B}
      \nabla_{\cBr_j} \obj = {}
      &
        2 \int_{\pset} \cbr_j(\overline{\pp})
        \xrd(\pp)
        \mybrack*{\yr(\pp) - \yf(\pp)}
        \difm{\pp},
      & j = 1, 2, \ldots, \qBr, \\*
      \label{eq:grad-C}
      \nabla_{\cCr_k} \obj = {}
      &
        2 \int_{\pset} \ccr_k(\overline{\pp})
        \mybrack*{\yr(\pp) - \yf(\pp)}
        \xr(\pp)\herm
        \difm{\pp},
      & k = 1, 2, \ldots, \qCr,
    \end{align}
  \end{subequations}
  where
  $\Ar(\pp)\herm \xrd(\pp) = \Cr(\pp)\herm$ is the reduced dual state equation,
  $\xrd(\pp) \in \CCro$ is the reduced dual state, and
  $(\cdot)\herm$ denotes the conjugate transpose.
\end{theorem}

Based on these gradients, in~\cite{MliG22}
we developed an $\Ltwo$-optimal reduced-order modeling algorithm and
demonstrated it on various examples,
both stationary parametric problems and \ac{lti} systems.
In this paper, we are more interested in the theoretical implications
of \Cref{thm:grad} than the algorithmic ones and show how it provides a
unifying framework for interpolatory optimal approximation.

\subsection{Necessary Conditions for \texorpdfstring{$\Ltwo$}{L2}-optimality}
An important consequence of \Cref{thm:grad} is that,
by setting the gradients to zero,
it yields the necessary optimality conditions for $\Ltwo$-optimal reduced-order
modeling using parameter-separable forms.
\begin{corollary}\label{cor:cond}
  Let $\pset$, $\measure$, $\yf$, $\car_i$, $\cbr_j$, $\ccr_k$, and $\cAr_i$
  satisfy \Cref{assumption:all}.
  Furthermore, let $(\cAr_i, \cBr_j, \cCr_k)$ be an $\Ltwo$-optimal \ac{strom}.
  Then
  \begin{subequations}\label{eq:cond}
    \begin{align}
      \label{eq:cond-C}
      \int_{\pset}
      \ccr_k(\overline{\pp})
      \yf(\pp) \xr(\pp)\herm
      \difm{\pp}
      & =
        \int_{\pset}
        \ccr_k(\overline{\pp})
        \yr(\pp) \xr(\pp)\herm
        \difm{\pp},
      & k = 1, 2, \ldots, \qCr, \\*
      \label{eq:cond-B}
      \int_{\pset}
      \cbr_j(\overline{\pp})
      \xr_d(\pp) \yf(\pp)
      \difm{\pp}
      & =
        \int_{\pset}
        \cbr_j(\overline{\pp})
        \xr_d(\pp) \yr(\pp)
        \difm{\pp},
      & j = 1, 2, \ldots, \qBr, \\*
      \label{eq:cond-A}
      \int_{\pset}
      \car_i(\overline{\pp})
      \xr_d(\pp) \yf(\pp) \xr(\pp)\herm
      \difm{\pp}
      & =
        \int_{\pset}
        \car_i(\overline{\pp})
        \xr_d(\pp) \yr(\pp) \xr(\pp)\herm
        \difm{\pp},
      & i = 1, 2, \ldots, \qAr.
    \end{align}
  \end{subequations}
\end{corollary}
The optimality conditions in \Cref{cor:cond} are
\emph{interpolatory} in the sense that the quantities (integrals)
in the left-hand sides of~\eqref{eq:cond-C}--\eqref{eq:cond-A}
involving the \ac{fom} output $\yf(\pp)$
need to be interpolated by the same integrals
involving the \ac{strom} output $\yr(\pp)$.
This result highlights that any $\Ltwo$-optimal \ac{strom} with the
parameter-separable form~\eqref{eq:rom-param-sep-form} is interpolatory in the
sense of \Cref{cor:cond}.
By carefully selecting, in~\eqref{eq:cond},
the scalar functions $\car_i, \cbr_j, \ccr_k$,
the parameter space $\pset$, and
the measure $\mu$ over $\pset$,
we derive concrete interpolatory optimality conditions
(in the form of bitangential Hermite interpolation)
for important classes of \acp{rom},
including non-parametric and parametric \ac{lti} systems,
stationary models, and
for discrete \ac{ls} problems,
thus providing a unifying framework for interpolatory $\Ltwo$-optimal
approximation across different domains.

\section{Linear Time-invariant Systems: Continuous Measure}%
\label{sec:lti}
Our goal in this section is to illustrate
(i) how \Cref{thm:grad} and \Cref{cor:cond} cover a wide range of
settings arising in optimal \ac{mor} of dynamical systems and
(ii) to develop new conditions for optimality.
Furthermore,
this analysis sets the stage for the interpolatory conditions we derive
in \Cref{sec:lti-ls} for discrete \ac{ls} minimization and
in \Cref{sec:stat} for stationary problems.

\subsection{\texorpdfstring{$\Htwo$}{H2}-optimal Model Order Reduction}%
\label{sec:lti-h2}
Interpolatory necessary optimality conditions are known for $\Htwo$-optimal
\ac{mor} of \ac{lti} systems,
both for the continuous-time case~\cite{MeiL67,GugAB08,AntBG10} and
the discrete-time case~\cite{BunKVetal10,GugAB08}.
In the following, we show that these conditions are a special case of the
conditions in \Cref{cor:cond}.

A continuous-time, finite-dimensional \ac{lti} system is given by
\begin{subequations}\label{eq:lti-cont-fom}
  \begin{align}
    E \dot{\xf}(t) & = A \xf(t) + B u(t), \quad \xf(0) = 0, \\*
    \yf(t) & = C \xf(t),
  \end{align}
\end{subequations}
where
$\xf(t) \in \RRf$ is the state,
$u(t) \in \RRi$ is the input,
$\yf(t) \in \RRo$ is the output,
$E, A \in \RRff$,
$B \in \RRfi$, and
$C \in \RRof$.
We assume that $E$ is invertible and all eigenvalues of $E^{-1} A$ have negative
real parts.
The rational function
\begin{align}\label{eq:lti-tf}
  H(s) = C {(s E - A)}^{-1} B
\end{align}
is the transfer function of~\eqref{eq:lti-cont-fom}
and satisfies $Y(s) = H(s) U(s)$
assuming $Y$ and $U$, the Laplace transforms of $\yf$ and $u$, exist;
see, e.g.,~\cite{Son98},
for the conditions for the existence of the Laplace transform.
Based on the assumptions above,
$H$ belongs to the $\Htwo^{\nout \times \nin}(\CC_+)$ Hardy space
(where $\CC_+$ is the open left half-plane),
which is the set of holomorphic functions $\fundef{F}{\CC_+}{\CCoi}$ such that
\(
  \sup_{\eta > 0}
  \int_{-\infty}^{\infty} \normF{F(\eta + \imag \omega)}^2 \dif{\omega}
  < \infty
\).
It is known that $F$ can be extended to $\overline{\CC_+}$ and the $\Htwo$ norm
can be defined as
\begin{equation*}
  \normHtwo{F}
  = \myparen*{
    \frac{1}{2 \pi}
    \int_{-\infty}^{\infty} \normF{F(\imag \omega)}^2 \dif{\omega}
  }^{1/2}.
\end{equation*}
The analysis applies to any $H \in \Htwo^{\nout \times \nin}(\CC_+)$ as
\iac{fom},
including infinite-dimensional systems, and
not just the finite-dimensional systems as in~\eqref{eq:lti-cont-fom}.

The goal of $\Htwo$-optimal \ac{mor} is to find \iac{rom}
\begin{subequations}\label{eq:lti-cont-rom}
  \begin{align}
    \hE \dot{\xr}(t) & = \hA \xr(t) + \hB u(t), \quad \xr(0) = 0, \\*
    \yr(t) & = \hC \xr(t),
  \end{align}
\end{subequations}
with
the reduced state $\xr(t) \in \RRr$,
the approximate output $\yr(t) \in \RRo$, and
the reduced quantities
$\hE, \hA \in \RRrr$,
$\hB \in \RRri$, and
$\hC \in \RRor$
such that its transfer function $\hH(s) = \hC (s \hE - \hA)^{-1} \hB$ minimizes
the $\Htwo$ error $\normHtwo{H - \hH}$.
The assumptions on the \ac{rom} are that $\hE$ be invertible and all eigenvalues
of $\hE^{-1} \hA$ have negative real parts.

To state the $\Htwo$-optimal interpolatory conditions,
let $\hH(s) = \hC (s \hE - \hA)^{-1} \hB$ have $\nrom$ distinct poles.
Let $\hT, \hS \in \CCrr$ be invertible matrices such that
$\hS\herm \hE \hT = I$ and
$\hS\herm \hA \hT = \Lambda$ with
$\Lambda = \mydiag{\lambda_1, \lambda_2, \ldots, \lambda_{\nrom}}$.
Then, we can write $\hH(s)$ as
\begin{align}\label{eq:pole-res-form}
  \hH(s)
  & =
    \hC \hT \myparen*{s I - \Lambda}^{-1} \hS\herm \hB
    =
    \sum_{j = 1}^{\nrom}
    \frac{\hC \hT e_j e_j\tran \hS\herm \hB}{s - \lambda_j}
    =
    \sum_{j = 1}^{\nrom} \frac{c_j b_j\herm}{s - \lambda_j},
\end{align}
where $c_j = \hC \hT e_j \in \CCo$, $b_j = \hB\tran \hS e_j\in \CCi$, and
$e_j$ denotes the $j$th unit vector (of appropriate size).
The formulation of $\hH$ in~\eqref{eq:pole-res-form} is called the
\emph{pole-residue form}
where $\lambda_j$ are the poles and $c_j b_j\herm$ are the (rank-$1$) residues.
If $\hH$ is an $\Htwo$-optimal \ac{rom} of $H$,
then it satisfies the interpolation conditions
\begin{subequations}\label{eq:h2-cond}
  \begin{align}
    \label{eq:h2-cond-c}
    H\myparen*{-\overline{\lambda_k}} b_k
    & =
      \hH\myparen*{-\overline{\lambda_k}} b_k, \\*
    \label{eq:h2-cond-b}
    c_k\herm H\myparen*{-\overline{\lambda_k}}
    & =
      c_k\herm \hH\myparen*{-\overline{\lambda_k}}, \\*
    \label{eq:h2-cond-a}
    c_k\herm H'\myparen*{-\overline{\lambda_k}} b_k
    & =
      c_k\herm \hH'\myparen*{-\overline{\lambda_k}} b_k,
  \end{align}
\end{subequations}
for $k = 1, 2, \ldots, \nrom$,
where $H'$ and $\hH'$ denote the derivatives with respect to $s$;
see~\cite{GugAB08,AntBG20}.
More specifically,
\eqref{eq:h2-cond-c} is called the right-tangential interpolation
condition,
\eqref{eq:h2-cond-b} the left-tangential interpolation condition, and
\eqref{eq:h2-cond-a} the bitangential Hermite interpolation condition.
We refer to all three conditions together as the bitangential Hermite
interpolation conditions for $\Htwo$-optimality.
They state that
the optimal reduced-order transfer function $\hH$
tangentially interpolates $H$ (and $\hH'$ interpolates $H'$)
at the mirror images of its own poles, i.e., at $-\overline{\lambda_k}$,
along the tangent directions $c_k$ and $b_k$
determined by its own rank-$1$ residues $c_k b_k\herm$.
These optimal interpolation conditions
have led to effective numerical methods for optimal \ac{mor} and
has been extended to various setting;
for details we refer the reader to~\cite{GugAB08,BeaG17,AntBG20} and
the references therein.
We also refer the reader to \cite{DruSZ14,DruS11,FenAB17,FenB19,FenB21} and the
references therein for greedy-based selections of interpolation points in
projection-based \ac{mor} of \ac{lti} systems.

Now we show how these optimal interpolatory conditions can be recovered from
\Cref{cor:cond} as a special case.
First we note that by setting
$\pset = \imag \RR$,
$\pp = \imag \omega$,
$\mu = \frac{1}{2 \pi} \lambda_{\imag \RR}$
(where $\lambda_{\imag \RR}$ is the Lebesgue measure over $\imag \RR$),
$\yf = H$, and
$\yr = \hH$,
we find $\normLtwomu{\yf} = \normHtwo{H}$ and
thus $\normLtwomu{\yf - \yr} = \normHtwo{H - \hH}$.
Furthermore, the reduced transfer function $\hH(s) = \hC (s \hE - \hA)^{-1} \hB$
can be viewed as \iac{strom} as in~\eqref{eq:rom} by rewriting it as
\begin{align*}
  \myparen*{s \hE - \hA} \hX(s) & = \hB, \\*
  \hH(s) & = \hC \hX(s),
\end{align*}
and by selecting
$\xr = \hX$,
$\yr = \hH$,
$\qAr = 2$,
$\car_1(\pp) = \pp$,
$\car_2(\pp) = -1$,
$\qBr = 1$,
$\cbr_1(\pp) = 1$,
$\qCr = 1$, and
$\ccr_1(\pp) = 1$
in~\eqref{eq:rom-param-sep-form}.
To obtain the optimality conditions~\eqref{eq:h2-cond} from
\Cref{cor:cond},
postmultiply the left-hand side of~\eqref{eq:cond-C} by $\hT\mherm e_k$
to get
\begin{align*}
  &
    \int_{\pset}
    \ccr_1(\overline{\pp}) \yf(\pp) \xr(\pp)\herm
    \difm{\pp}
    \, \hT\mherm e_k
    =
    \frac{1}{2 \pi}
    \int_{-\infty}^{\infty}
    H(\imag \omega)
    \hB\tran \myparen*{\imag \omega \hE - \hA}\mherm
    \dif{\omega} \, \hT\mherm e_k \\
  & =
    \frac{1}{2 \pi}
    \int_{-\infty}^{\infty}
    H(\imag \omega)
    \hB\tran \hS \myparen*{\imag \omega I - \Lambda}\mherm \hT\herm
    \dif{\omega} \, \hT\mherm e_k
    =
    \frac{1}{2 \pi}
    \int_{-\infty}^{\infty}
    \frac{H(\imag \omega) b_k}{-\imag \omega - \overline{\lambda_k}}
    \dif{\omega}.
\end{align*}
Switching to a contour integral with $s = \imag \omega$ and
$\dif{s} = \imag \dif{\omega}$,
we find
\begin{align*}
  &
    \int_{\pset}
    \ccr_1(\overline{\pp}) \yf(\pp) \xr(\pp)\herm
    \difm{\pp}
    \, \hT\mherm e_k
    =
    \frac{1}{2 \pi \imag}
    \oint_{\imag \RR}
    \frac{H(s) b_k}{-s - \overline{\lambda_k}}
    \dif{s}
    =
    -\frac{1}{2 \pi \imag}
    \oint_{\imag \RR}
    \frac{H(s) b_k}{s - (-\overline{\lambda_k})}
    \dif{s} \\
  & =
    \lim_{R \to \infty}
    -\frac{1}{2 \pi \imag}
    \oint_{\Gamma_R}
    \frac{H(s) b_k}{s - (-\overline{\lambda_k})}
    \dif{s}
    =
    H\myparen*{-\overline{\lambda_k}} b_k,
\end{align*}
where
$\Gamma_R
= [-\imag R, \imag R]
\cup \{R e^{\imag \omega} : \omega \in [-\frac{\pi}{2}, \frac{\pi}{2}]\}$
is the clockwise contour of a semidisk as in~\cite[Lemma~1.1]{AntBG10}.
Applying the same manipulations to the right-hand side of~\eqref{eq:cond-C}
yields $\hH\myparen{-\overline{\lambda_k}} b_k$.
Thus the first optimality conditions~\eqref{eq:cond-C} in \Cref{cor:cond}
yield the right-tangential interpolation conditions~\eqref{eq:h2-cond-c}
for $\Htwo$-optimality as a special case.

Similarly,
premultiplying the left-hand side of~\eqref{eq:cond-B} by $e_k\tran \hS^{-1}$,
we obtain
\begin{align*}
  e_k\tran \hS^{-1}
  \int_{\pset}
  \cbr_1(\overline{\pp}) \xr_d(\pp) \yf(\pp)
  \difm{\pp}
  =
  c_k\herm H\myparen*{-\overline{\lambda_k}}.
\end{align*}
Thus, \eqref{eq:cond-B} yields the left-tangential
interpolation condition~\eqref{eq:h2-cond-b} for $\Htwo$-optimality
(after applying the same manipulations to the right-hand side
of~\eqref{eq:cond-B}).
Finally, taking the left-hand side of~\eqref{eq:cond-A} in
\Cref{cor:cond} related to $\hA$ with $\car_2(\pp) = -1$,
premultiplying it by $e_k\tran \hS^{-1}$, and
postmultiplying it by $\hT\mherm e_k$
gives
\begin{align*}
  &
    e_k\tran \hS^{-1}
    \int_{\pset} \car_2(\overline{\pp}) \xr_d(\pp) \yf(\pp) \xr(\pp)\herm
    \difm{\pp}
    \, \hT\mherm e_k
    =
    -\frac{1}{2 \pi}
    \int_{-\infty}^{\infty}
    \frac{c_k\herm H(\imag \omega) b_k}%
    {\myparen*{-\imag \omega - \overline{\lambda_k}}^2}
    \dif{\omega} \\
  & =
    -\frac{1}{2 \pi \imag}
    \oint_{\imag \RR}
    \frac{c_k\herm H(s) b_k}{\myparen*{-s - \overline{\lambda_k}}^2}
    \dif{s}
    =
    c_k\herm H'\myparen*{-\overline{\lambda_k}} b_k,
\end{align*}
obtaining the final Hermite interpolation condition~\eqref{eq:h2-cond-a}.
Thus, we recover the interpolatory $\Htwo$-optimality
conditions~\eqref{eq:h2-cond}
as a special case of the more general $\Ltwo$-optimality
conditions~\eqref{eq:cond} in \Cref{cor:cond}.

\begin{remark}
  The case of discrete-time \ac{lti} systems follows similarly.
  Discrete-time systems are obtained by replacing the derivative term
  $\dot{x}(t)$ in~\eqref{eq:lti-cont-fom} with the time shift $x(t + 1)$ and
  restricting the time $t$ to integers.
  Transfer function $H$ of a discrete-time \ac{lti} system has exactly the same
  form as~\eqref{eq:lti-tf},
  but we now assume that $E^{-1} A$ has eigenvalues in the open unit disk.
  The corresponding Hardy space is
  $\htwo^{\nout \times \nin}(\overline{\DD}^c)$
  ($\DD$ is the open unit disk and
  $\overline{\DD}^c$ is the complement of the closed unit disk)
  containing functions $F$ such that
  \(
    \sup_{r > 1}
    \int_{0}^{2 \pi} \normF{F(r e^{\imag \omega})}^2 \dif{\omega}
    < \infty
  \)
  and the $\htwo$ norm is given by
  \begin{equation*}
    \normhtwo{F}
    = \myparen*{
      \frac{1}{2 \pi}
      \int_{0}^{2 \pi} \normF*{F(e^{\imag \omega})}^2 \dif{\omega}
    }^{1/2}.
  \end{equation*}
  Then, we recover the interpolatory necessary optimality conditions
  from~\cite{BunKVetal10,GugAB08} analogously to the continuous-time case,
  in this case by setting
  $\pset = \partial\DD$,
  $\pp = e^{\imag \omega}$,
  $\mu = \frac{1}{2 \pi} \lambda_{\partial\DD}$
  (where $\lambda_{\partial\DD}$ is the Lebesgue measure over $\partial\DD$),
  and taking the contour integral over the unit circle.
  In particular, if
  $\hH(s) = \sum_{j = 1}^{\nrom} \frac{c_j b_j\herm}{s - \lambda_j}$ is an
  $\htwo$-optimal \ac{rom} for $H$,
  then
  \begin{equation}\label{eq:h2-cond-dt}
    H\myparen*{\frac{1}{\overline{\lambda_k}}} b_k
    = \hH\myparen*{\frac{1}{\overline{\lambda_k}}} b_k,\ \
    c_k\herm H\myparen*{\frac{1}{\overline{\lambda_k}}}
    = c_k\herm \hH\myparen*{\frac{1}{\overline{\lambda_k}}},\ \
    c_k\herm H'\myparen*{\frac{1}{\overline{\lambda_k}}} b_k
    = c_k\herm \hH'\myparen*{\frac{1}{\overline{\lambda_k}}} b_k,
  \end{equation}
  for $k = 1, 2, \ldots, \nrom$.
  As in the continuous-time case,
  bitangential Hermite interpolation forms the necessary conditions for
  $\htwo$-optimality
  where the interpolation points depend on the reduced-order poles and
  the tangent directions on the reduced-order residues.
  The difference is that the mirroring of the reduced-order poles to the
  interpolation points is now done with respect to the unit circle.
\end{remark}

\subsection{\texorpdfstring{$\HtwoLtwo$}{H2xL2}-optimal Parametric Model Order
  Reduction}%
\label{sec:grimm}
In \Cref{sec:lti-h2},
we considered non-parametric \ac{lti} systems as in~\eqref{eq:lti-cont-fom} and
showed how \Cref{cor:cond} recovers the well-known $\Htwo$-optimality
conditions.
In this section,
we focus on parametric \ac{lti} systems
where the underlying dynamics depend on a set of parameters and
thus consider \emph{jointly optimal} approximation in the frequency and
parameter space.
We show how our $\Ltwo$-optimal modeling framework naturally covers this problem
as well and
extends the existing interpolatory necessary optimality conditions to a more
general setting.

We consider \acp{fom} of the form
\begin{subequations}\label{eq:plti-fom}
  \begin{align}
    \Ef(\ppp) \dot{\xf}(t, \ppp)
    & = \Af(\ppp) \xf(t, \ppp) + \Bf(\ppp) u(t),
      \quad \xf(0, \ppp) = 0, \\*
    \yf(t, \ppp)
    & = \Cf(\ppp) \xf(t, \ppp),
  \end{align}
\end{subequations}
where
$\ppp \in \psetp \subset \CC$ is the parameter,
$\psetp$ is the parameter space,
$\xf(t, \ppp) \in \RRf$ is the state,
$u(t) \in \RR$ is the input,
$\yf(t, \ppp) \in \RR$ is the output,
$\Ef(\ppp), \Af(\ppp) \in \RRff$,
$\Bf(\ppp) \in \RRfi$, and
$\Cf(\ppp) \in \RRof$.
In practical applications, the variable $\ppp$ can correspond to, e.g.,
geometry, material properties (such as thickness), boundary conditions etc.
We assume that
$\Ef, \Af, \Bf, \Cf$ are holomorphic,
$\Ef(\ppp)$ is invertible for all $\ppp \in \psetp$, and
$\Ef(\ppp)^{-1} \Af(\ppp)$ has eigenvalues in the open left half-plane for all
$\ppp \in \psetp$.
Then the corresponding parametric transfer function of~\eqref{eq:plti-fom} is
$H(s, \ppp) = \Cf(\ppp) (s \Ef(\ppp) - \Af(\ppp))^{-1} \Bf(\ppp)$.
Note that unlike in the non-parametric \ac{lti} case,
transfer function now depends on both the frequency variable $s$ and the
parameter $\ppp$.
Thus a reduced-order transfer-function approximation
$\hH(s, \ppp) = \Cr(\ppp) (s \Er(\ppp) - \Ar(\ppp))^{-1} \Br(\ppp)$
to $H(s, \ppp)$ should have high fidelity both in $s$ and $\ppp$.
How should one choose the space in which to approximate $H(s,\ppp)$?

Grimm~\cite{Gri18} focused on a simplified problem and considered the
single-input and single-output parametric \ac{lti} system, i.e.,
$\nin = \nout = 1$ in~\eqref{eq:plti-fom}, and
as the approximation space considered the Hardy space in two variables
$\Htwo(\CC_+ \times \DD)$,
which is the space of holomorphic functions $\fundef{F}{\CC_+ \times \DD}{\CC}$
such that
\begin{equation*}
  \sup_{\eta > 0,\ 0 < r < 1}
  \int_{0}^{2 \pi}
  \int_{-\infty}^{\infty}
  \abs*{F(\eta + \imag \omega, r e^{\imag \omega_{\ppp}})}^2
  \dif{\omega}
  \dif{\omega_{\ppp}}
  < \infty.
\end{equation*}
The corresponding $\Htwo(\CC_+ \times \DD)$ norm,
refereed to as the $\HtwoLtwo$ norm,
is given by
\begin{equation*}
  \normHtwoLtwo{F}
  =
  \myparen*{
    \frac{1}{4 \pi^2}
    \int_{0}^{2 \pi}
    \int_{-\infty}^{\infty}
    \abs*{F(\imag \omega, e^{\imag \omega_{\ppp}})}^2
    \dif{\omega}
    \dif{\omega_{\ppp}}
  }^{1/2}.
\end{equation*}
The $\HtwoLtwo$~used in~\cite{Gri18} is a special case of the more general
$\HtwoLtwo$ norm definition introduced in~\cite{BauBBetal11}.
Unlike the non-parametric \ac{lti} case considered in \Cref{sec:lti-h2},
the optimal interpolatory conditions for \ac{mor} in the
$\HtwoLtwo$ norm are not known in general except for some special cases.
For example, \cite{BauBBetal11} shows that if the parametric dependency $\ppp$
only appears in $\Bf(\ppp)$ and $\Cf(\ppp)$ in~\eqref{eq:plti-fom}, and
$\Af$ and $\Ef$ are non-parametric,
then one can convert the resulting problem into an equivalent $\Htwo$-optimal
approximation problem and
obtain interpolatory optimality conditions.
However, this is restrictive since in most parametric problems,
$\Af$ and $\Ef$ vary with a parameter.

Grimm~\cite{Gri18}, instead,
considers a simplification in the form of the \ac{rom}.
Recall that the pole-residue form~\eqref{eq:pole-res-form} has proved vital in
deriving the $\Htwo$-optimality conditions.
Inspired by~\eqref{eq:pole-res-form},
\cite{Gri18} considered \acp{rom} with the form
\begin{equation}\label{eq:grimm-Hhat}
  \hH(s, \ppp)
  = \sum_{i = 1}^{\nrom} \sum_{j = 1}^{\nrom_{\ppp}}
  \frac{\phi_{ij}}{(s - \lambda_i) (\ppp - \pi_j)},
\end{equation}
where
$\phi_{ij} \in \CC$,
$\lambda_i \in \CC_-$, and
$\pi_j \in \overline{\DD}^c$
($\CC_-$ is the open left half-plane).
In~\eqref{eq:grimm-Hhat},
one may view $\lambda_i$'s as the frequency poles (in the $s$ variable) and
$\pi_j$'s as the parameter poles (in the $\ppp$ variable).
Then~\cite[Thm~3.3.4]{Gri18} shows that
the interpolatory necessary conditions for $\hH$ in~\eqref{eq:grimm-Hhat} to be
an $\HtwoLtwo$-optimal \ac{rom} are
\begin{subequations}\label{eq:grimm-cond}
  \begin{align}
    \label{eq:grimm-cond-a}
    H\myparen*{-\lambda_k, \frac{1}{\pi_{\ell}}}
    & =
      \hH\myparen*{-\lambda_k, \frac{1}{\pi_{\ell}}}, \\
    \label{eq:grimm-cond-b}
    \sum_{j = 1}^{\nrom_{\ppp}}
    \frac{\phi_{kj}}{\pi_j}
    \frac{\partial H}{\partial s}
    \myparen*{-\lambda_k, \frac{1}{\pi_j}}
    & =
      \sum_{j = 1}^{\nrom_{\ppp}}
      \frac{\phi_{kj}}{\pi_j}
      \frac{\partial \hH}{\partial s}
      \myparen*{-\lambda_k, \frac{1}{\pi_j}}, \\
    \label{eq:grimm-cond-c}
    \sum_{i = 1}^{\nrom}
    \phi_{i \ell}
    \frac{\partial H}{\partial \ppp}
    \myparen*{-\lambda_i, \frac{1}{\pi_{\ell}}}
    & =
      \sum_{i = 1}^{\nrom}
      \phi_{i \ell}
      \frac{\partial \hH}{\partial \ppp}
      \myparen*{-\lambda_i, \frac{1}{\pi_{\ell}}},
  \end{align}
\end{subequations}
for $k = 1, 2, \ldots, \nrom$ and $\ell = 1, 2, \ldots, \nrom_{\ppp}$.
In the rest of this section, we show that the interpolation conditions in
\Cref{cor:cond} cover the $\HtwoLtwo$ framework of~\cite{Gri18} and
at the same time extend the analysis to $\HtwoLtwo$ approximation of
\emph{multiple-input/multiple-output} parametric \ac{lti} systems.
Therefore, in~\eqref{eq:plti-fom} we do not need to assume $\nin = \nout = 1$ as
done in~\cite{Gri18}.

For matrix-valued transfer functions $H$ and $\hH$,
the squared $\HtwoLtwo$ error is defined as
\begin{equation}\label{eq:grimm-obj}
  \normHtwoLtwo*{H - \hH}^2
  =
  \frac{1}{4 \pi^2}
  \int_{0}^{2 \pi} \int_{-\infty}^{\infty}
  \normF*{
    H\myparen*{\imag \omega, e^{\imag \omega_{\ppp}}}
    - \hH\myparen*{\imag \omega, e^{\imag \omega_{\ppp}}}
  }^2
  \dif{\omega}
  \dif{\omega_{\ppp}}.
\end{equation}
Observe that
the \ac{rom}~\eqref{eq:grimm-Hhat} (for the single-input/single-output case) can
be written as \iac{strom} in~\eqref{eq:rom} using the representation
\begin{align*}
  \mybrack*{(s I_{\nrom} - \Lambda) \otimes (\ppp I_{\nrom_{\ppp}} - \Pi)}
  \hX(s, \ppp)
  & = \ones, \\*
  \hH(s, \ppp)
  & = \phi\tran \hX(s, \ppp),
\end{align*}
where
$\Lambda = \mydiag{\lambda_1, \lambda_2, \ldots, \lambda_{\nrom}}$,
$\Pi = \mydiag{\pi_1, \pi_2, \ldots, \pi_{\nrom_{\ppp}}}$,
$\ones = [1, \ldots, 1]\tran$, and
$\phi = [\phi_{1, 1}, \ldots, \phi_{\nrom, \nrom_{\ppp}}]\tran$.
To extend this formulation to multiple-input/multiple-output problems,
we then consider
\begin{subequations}\label{eq:grimm-rom}
  \begin{align}
    \mybrack*{\myparen*{s \hE - \hA}
    \otimes \myparen*{\ppp \hE_{\ppp} - \hA_{\ppp}}}
    \hX(s, \ppp)
    & = \hB, \\*
    \hH(s, \ppp)
    & = \hC \hX(s, \ppp),
  \end{align}
\end{subequations}
where
$\hE, \hA \in \RRrr$,
$\hE_{\ppp}, \hA_{\ppp} \in \RR^{\nrom_{\ppp} \times \nrom_{\ppp}}$,
$\hB \in \RR^{\nrom \nrom_{\ppp} \times \nin}$, and
$\hC \in \RR^{\nout \times \nrom \nrom_{\ppp}}$.
Note that~\eqref{eq:grimm-rom} is \iac{strom} as in~\eqref{eq:rom} with
$\pp = (s, \ppp)$ and
$\Ar(\pp)
= \myparen{s \hE - \hA}
\otimes \myparen{\ppp \hE_{\ppp} - \hA_{\ppp}}$.
By expanding $\Ar(\pp)$,
we find that the form~\eqref{eq:grimm-rom} corresponds to $\qAr = 4$ and
\begin{subequations}\label{eq:KronstROM}
  \begin{alignat}{8}
    & \car_1(s, \ppp) = s \ppp, & \quad
    & \car_2(s, \ppp) = -s, & \quad
    & \car_3(s, \ppp) = -\ppp, & \quad
    & \car_4(s, \ppp) = 1, \\*
    & \cAr_1 = \hE \otimes \hE_{\ppp}, & \quad
    & \cAr_2 = \hE \otimes \hA_{\ppp}, & \quad
    & \cAr_3 = \hA \otimes \hE_{\ppp}, & \quad
    & \cAr_4 = \hA \otimes \hA_{\ppp}.
  \end{alignat}
\end{subequations}
Therefore, the \ac{rom}~\eqref{eq:grimm-rom} is \iac{strom}~\eqref{eq:rom}
fitting into our $\Ltwo$-optimal modeling framework
where the \ac{strom} matrices have an additional Kronecker structure.
The following lemma will help us in computing the gradients of the cost function
with respect to $\hA$ and $\hA_{\ppp}$ having this specific Kronecker structure.
\begin{lemma}\label{lem:kron-grad}
  Let $\fundef{F}{\RR^{n m \times n m}}{\RR}$ be a differentiable function at
  $X = A \otimes B$ where $A \in \RR^{n \times n}$ and $B \in \RR^{m \times m}$
  are nonzero matrices.
  Let $\fundef{G}{\RR^{n \times n}}{\RR}$ be defined as
  $G(A_1) = F(A_1 \otimes B)$.
  Then $G$ is differentiable at $A$ and
  \begin{equation*}
    \nabla G(A)
    =
    \sum_{j = 1}^m
    \myparen*{I_n \otimes e_j\tran B_L\herm}
    \nabla F(X)
    \myparen*{I_n \otimes B_R\herm e_j},
  \end{equation*}
  for any $B_L, B_R \in \CC^{m \times m}$ such that $B = B_L B_R$.
  Similarly, let $\fundef{H}{\RR^{m \times m}}{\RR}$ be defined as
  $H(B_1) = F(A \otimes B_1)$.
  Then $H$ is differentiable at $B$ and
  \begin{equation*}
    \nabla H(B)
    =
    \sum_{i = 1}^n
    \myparen*{e_i\tran A_L\herm \otimes I_m}
    \nabla F(X)
    \myparen*{A_R\herm e_i \otimes I_m},
  \end{equation*}
  for any $A_L, A_R \in \CC^{n \times n}$ such that $A = A_L A_R$.
\end{lemma}
\begin{proof}
  See \Cref{sec:lemma-proof}.
\end{proof}
Before stating the next theorem, we need to derive a pole-residue form
of~\eqref{eq:grimm-rom}.
Let $\hT, \hS \in \CCrr$ and
$\hT_{\ppp}, \hS_{\ppp} \in \CC^{\nrom_{\ppp} \times \nrom_{\ppp}}$
be invertible matrices such that
$\hS\herm \hE \hT = I_{\nrom}$,
$\hS\herm \hA \hT = \Lambda$,
$\hS_{\ppp}\herm \hE_{\ppp} \hT_{\ppp} = I_{\nrom_{\ppp}}$, and
$\hS_{\ppp}\herm \hA_{\ppp} \hT_{\ppp} = \Pi$.
Then
\begin{align}
  \nonumber
  \hH(s,\ppp) = {}
  &
    \hC
    \myparen*{\hT \otimes \hT_{\ppp}}
    \mybrack*{
      \myparen*{s I_{\nrom} - \Lambda}
      \otimes \myparen*{\ppp I_{\nrom_{\ppp}} - \Pi}
    }^{-1}
    \myparen*{\hS\herm \otimes \hS_{\ppp}\herm}
    \hB \\*
  \label{eq:plti-rom-pole-res-form}
  = {}
  &
    \sum_{i = 1}^{\nrom}
    \sum_{j = 1}^{\nrom_{\ppp}}
    \frac{c_{ij} b_{ij}\herm}{(s - \lambda_i) (\ppp - \pi_j)},
\end{align}
where $c_{ij} = \hC (\hT \otimes \hT_{\ppp}) (e_i \otimes e_j)$ and
$b_{ij} = \hB\tran (\hS \otimes \hS_{\ppp}) (e_i \otimes e_j)$.

\begin{theorem}
  Let $\hH$ be a \ac{strom} as in~\eqref{eq:grimm-rom} with the pole-residue
  form~\eqref{eq:plti-rom-pole-res-form}
  where $\lambda_i \in \CC_-$ and $\pi_j \in \overline{\DD}^c$ are pairwise
  distinct.
  If $\hH$ is an $\HtwoLtwo$-optimal approximation to
  $H \in \Htwo^{\nout \times \nin}(\CC_+ \times \DD)$, then
  \begin{subequations}\label{eq:grimm-mimo-cond}
    \begin{align}
      \label{eq:grimm-mimo-cond-1}
      H\myparen*{-\overline{\lambda_k}, \frac{1}{\overline{\pi_{\ell}}}}
      b_{k \ell}
      & =
        \hH\myparen*{-\overline{\lambda_k}, \frac{1}{\overline{\pi_{\ell}}}}
        b_{k \ell}, \\
      \label{eq:grimm-mimo-cond-2}
      c_{k \ell}\herm
      H\myparen*{-\overline{\lambda_k}, \frac{1}{\overline{\pi_{\ell}}}}
      & =
        c_{k \ell}\herm
        \hH\myparen*{-\overline{\lambda_k}, \frac{1}{\overline{\pi_{\ell}}}}, \\
      \label{eq:grimm-mimo-cond-3}
      \sum_{j = 1}^{\nrom_{\ppp}}
      \frac{1}{\overline{\pi_j}}
      c_{kj}\herm
      \frac{\partial H}{\partial s}
      \myparen*{-\overline{\lambda_k}, \frac{1}{\overline{\pi_j}}}
      b_{kj}
      & =
        \sum_{j = 1}^{\nrom_{\ppp}}
        \frac{1}{\overline{\pi_j}}
        c_{kj}\herm
        \frac{\partial \hH}{\partial s}
        \myparen*{-\overline{\lambda_k}, \frac{1}{\overline{\pi_j}}}
        b_{kj}, \\
      \label{eq:grimm-mimo-cond-4}
      \sum_{i = 1}^{\nrom}
      c_{i \ell}\herm
      \frac{\partial H}{\partial \ppp}
      \myparen*{-\overline{\lambda_i}, \frac{1}{\overline{\pi_{\ell}}}}
      b_{i \ell}
      & =
        \sum_{i = 1}^{\nrom}
        c_{i \ell}\herm
        \frac{\partial \hH}{\partial \ppp}
        \myparen*{-\overline{\lambda_i}, \frac{1}{\overline{\pi_{\ell}}}}
        b_{i \ell},
    \end{align}
  \end{subequations}
  for $k = 1, 2, \ldots, \nrom$ and $\ell = 1, 2, \ldots, \nrom_{\ppp}$.
\end{theorem}
\begin{proof}
  First observe that we recover the $\HtwoLtwo$ norm by setting
  $\pset = \imag \RR \times \partial\DD$,
  $\pp = (\imag \omega, e^{\imag \omega_{\ppp}})$,
  $\mu = \frac{1}{4 \pi^2} \lambda_{\imag \RR} \times \lambda_{\partial\DD}$,
  and $\yf(\pp) = H(s, \ppp)$
  in the definition of the $\Ltwo(\pset, \measure)$ norm.
  Also we have already shown that $\hH$ in~\eqref{eq:grimm-rom} is \iac{strom}
  as in~\eqref{eq:rom}
  with $\xr(\pp) = \hX(s, \ppp)$, $\yr(\pp) = \hH(s, \ppp)$, and
  the Kronecker structure~\eqref{eq:KronstROM}.
  We start by postmultiplying the left-hand side of the first
  condition~\eqref{eq:cond-C} in \Cref{cor:cond} by
  $\myparen{\hT \otimes \hT_{\ppp}}\mherm (e_k \otimes e_{\ell})$
  to obtain
  \begin{align*}
    &
      \int_{\pset} \yf(\pp) \xr(\pp)\herm \difm{\pp}
      \, \myparen*{\hT \otimes \hT_{\ppp}}\mherm (e_k \otimes e_{\ell}) \\*
    & =
      \frac{1}{4 \pi^2}
      \int_0^{2 \pi}
      \int_{-\infty}^{\infty}
      \frac{H(\imag \omega, e^{\imag \omega_{\ppp}}) b_{k \ell}}{
        \myparen*{-\imag \omega - \overline{\lambda_k}}
        \myparen*{e^{-\imag \omega_{\ppp}} - \overline{\pi_{\ell}}}
      }
      \dif{\omega}
      \dif{\omega_{\ppp}}.
  \end{align*}
  Substituting
  $s = \imag \omega$,
  $\dif{s} = \imag \dif{\omega}$,
  $\ppp = e^{\imag \omega_{\ppp}}$, and
  $\dif{\ppp} = \imag \ppp \dif{\omega_{\ppp}}$,
  we find
  \begin{align*}
    &
      \int_{\pset} \yf(\pp) \xr(\pp)\herm \difm{\pp}
      \, \myparen*{\hT \otimes \hT_{\ppp}}\mherm (e_k \otimes e_{\ell})
      =
      \frac{1}{4 \pi^2 \imag^2}
      \oint_{\partial\DD}
      \oint_{\imag \RR}
      \frac{\frac{1}{\ppp} H(s, \ppp) b_{k \ell}}{
        \myparen*{-s - \overline{\lambda_k}}
        \myparen*{\frac{1}{\ppp} - \overline{\pi_{\ell}}}
      }
      \dif{s}
      \dif{\ppp} \\
    & =
      \frac{1}{4 \pi^2 \imag^2 \overline{\pi_{\ell}}}
      \oint_{\partial\DD}
      \oint_{\imag \RR}
      \frac{H(s, \ppp) b_{k \ell}}{
        \myparen*{s - \myparen*{-\overline{\lambda_k}}}
        \myparen*{\ppp - \frac{1}{\overline{\pi_{\ell}}}}
      }
      \dif{s}
      \dif{\ppp}
      =
      -\frac{1}{\overline{\pi_{\ell}}}
      H\myparen*{-\overline{\lambda_k}, \frac{1}{\overline{\pi_{\ell}}}}
      b_{k \ell},
  \end{align*}
  by using the Cauchy integral formula twice in the last equality.
  Performing similar operations to the right-hand side of~\eqref{eq:cond-C},
  we obtain the condition~\eqref{eq:grimm-mimo-cond-1}.
  Similarly,
  premultiplying the left-hand side of~\eqref{eq:cond-B} by
  $(e_k \otimes e_{\ell})\tran \myparen{\hS \otimes \hS_{\ppp}}^{-1}$,
  we obtain
  \begin{align*}
    (e_k \otimes e_{\ell})\tran \myparen*{\hS \otimes \hS_{\ppp}}^{-1}
    \int_{\pset} \xr_d(\pp) \yf(\pp) \difm{\pp}
    =
    -\frac{1}{\overline{\pi_{\ell}}}
    c_{k \ell}\herm
    H\myparen*{-\overline{\lambda_k}, \frac{1}{\overline{\pi_{\ell}}}}.
  \end{align*}
  Doing the same for the right-hand side of~\eqref{eq:cond-B},
  we obtain the condition~\eqref{eq:grimm-mimo-cond-2}.

  Similar to recovering the bitangential Hermite condition for
  $\Htwo$-optimality~\eqref{eq:h2-cond-a}
  where we used the gradient of the squared $\Htwo$ error with respect to $\hA$,
  in this setting, we need to differentiate the squared $\HtwoLtwo$
  error~\eqref{eq:grimm-obj} with respect to $\hA$ and $\hA_\ppp$.
  We start by computing the gradient with respect to $\hA$.
  Using \Cref{lem:kron-grad} with
  $\hE_{\ppp} = \hS_{\ppp}\mherm \hT_{\ppp}^{-1}$ and
  $\hA_{\ppp} = (\hS_{\ppp}\mherm \Pi) \hT_{\ppp}^{-1}$,
  we see that
  \begin{align*}
    &
      \frac{1}{2}
      \nabla_{\hA} \normHtwoLtwo*{H - \hH}^2 \\*
    & =
      \sum_{j = 1}^{\nrom_{\ppp}}
      \myparen*{I_{\nrom} \otimes e_j\tran \hS_{\ppp}^{-1}}
      \int_{\pset}
      \car_3(\overline{\pp})
      \xrd(\pp)
      \mybrack*{\yf(\pp) - \yr(\pp)}
      \xr(\pp)\herm
      \difm{\pp}
      \myparen*{I_{\nrom} \otimes \hT_{\ppp}\mherm e_j} \\*
    & \quad
      + \sum_{j = 1}^{\nrom_{\ppp}}
      \myparen*{I_{\nrom} \otimes e_j\tran \Pi\herm \hS_{\ppp}^{-1}}
      \int_{\pset}
      \car_4(\overline{\pp})
      \xrd(\pp)
      \mybrack*{\yf(\pp) - \yr(\pp)}
      \xr(\pp)\herm
      \difm{\pp}
      \myparen*{I_{\nrom} \otimes \hT_{\ppp}\mherm e_j} \\
    & =
      -\sum_{j = 1}^{\nrom_{\ppp}}
      \int_{\pset}
      \myparen*{\overline{\ppp} - \overline{\pi_j}}
      \myparen*{I_{\nrom} \otimes e_j\tran \hS_{\ppp}^{-1}}
      \xrd(\pp)
      \mybrack*{\yf(\pp) - \yr(\pp)}
      \xr(\pp)\herm
      \myparen*{I_{\nrom} \otimes \hT_{\ppp}\mherm e_j}
      \difm{\pp}.
  \end{align*}
  Then,
  \begin{align*}
    &
      \frac{1}{2}
      \myparen*{e_k\tran \hS^{-1} \otimes I_{\nrom_{\ppp}}}
      \nabla_{\hA} \normHtwoLtwo*{H - \hH}^2
      \myparen*{\hT\mherm e_k \otimes I_{\nrom_{\ppp}}} \\*
    & =
      -\frac{1}{4 \pi^2}
      \sum_{j = 1}^{\nrom_{\ppp}}
      \int_0^{2 \pi}
      \int_{-\infty}^{\infty}
      \myparen*{e^{-\imag \omega_{\ppp}} - \overline{\pi_j}}
      \frac{
        c_{k j}\herm
        \myparen*{
          H(\imag \omega, e^{\imag \omega_{\ppp}})
          - \hH(\imag \omega, e^{\imag \omega_{\ppp}})
        }
        b_{k j}
      }{
        \myparen*{-\imag \omega - \overline{\lambda_k}}^2
        \myparen*{e^{-\imag \omega_{\ppp}} - \overline{\pi_j}}^2
      }
      \dif{\omega}
      \dif{\omega_{\ppp}} \\
    & =
      -\frac{1}{4 \pi^2}
      \sum_{j = 1}^{\nrom_{\ppp}}
      \int_0^{2 \pi}
      \int_{-\infty}^{\infty}
      \frac{
        c_{k j}\herm
        \myparen*{
          H(\imag \omega, e^{\imag \omega_{\ppp}})
          - \hH(\imag \omega, e^{\imag \omega_{\ppp}})
        }
        b_{k j}
      }{
        \myparen*{-\imag \omega - \overline{\lambda_k}}^2
        \myparen*{e^{-\imag \omega_{\ppp}} - \overline{\pi_j}}
      }
      \dif{\omega}
      \dif{\omega_{\ppp}} \\
    & =
      \frac{1}{4 \pi^2 \imag^2}
      \sum_{j = 1}^{\nrom_{\ppp}}
      \frac{1}{\overline{\pi_j}}
      \oint_{\partial\DD}
      \oint_{\imag \RR}
      \frac{
        c_{k j}\herm
        \myparen*{H(s, \ppp) - \hH(s, \ppp)}
        b_{k j}
      }{
        \myparen*{s - \myparen*{-\overline{\lambda_k}}}^2
        \myparen*{\ppp - \frac{1}{\overline{\pi_j}}}
      }
      \dif{s}
      \dif{\ppp} \\
    & =
      -\sum_{j = 1}^{\nrom_{\ppp}}
      \frac{1}{\overline{\pi_j}}
      c_{k j}\herm
      \myparen*{
        \frac{\partial H}{\partial s}\myparen*{-\overline{\lambda_k},
          \frac{1}{\overline{\pi_j}}}
        - \frac{\partial \hH}{\partial s}\myparen*{-\overline{\lambda_k},
          \frac{1}{\overline{\pi_j}}}
      }
      b_{k j},
  \end{align*}
  where we used the Cauchy integral formula twice in the last equality.
  Setting this equation equal to zero gives the
  condition~\eqref{eq:grimm-mimo-cond-3}.

  Lastly, computing the gradient with respect to $\hA_{\ppp}$
  using \Cref{lem:kron-grad} with
  $\hE = \hS\mherm \hT^{-1}$ and
  $\hA = (\hS\mherm \Lambda) \hT^{-1}$,
  we obtain (similar to the gradient with respect to $\hA$)
  \begin{align*}
    &
      \frac{1}{2}
      \nabla_{\hA_{\ppp}} \normHtwoLtwo*{H - \hH}^2 \\*
    & =
      -\sum_{i = 1}^{\nrom}
      \int_{\pset}
      \myparen*{\overline{s} - \overline{\lambda_i}}
      \myparen*{e_i\tran \hS^{-1} \otimes I_{\nrom_{\ppp}}}
      \xrd(\pp)
      \mybrack*{\yf(\pp) - \yr(\pp)}
      \xr(\pp)\herm
      \myparen*{\hT\mherm e_i \otimes I_{\nrom_{\ppp}}}
      \difm{\pp}
  \end{align*}
  and
  \begin{align*}
    &
      \frac{1}{2}
      \myparen*{I_{\nrom} \otimes e_{\ell}\tran \hS_{\ppp}^{-1}}
      \nabla_{\hA_{\ppp}} \normHtwoLtwo*{H - \hH}^2
      \myparen*{I_{\nrom} \otimes \hT_{\ppp}\mherm e_{\ell}} \\*
    & =
      -\frac{1}{4 \pi^2}
      \sum_{i = 1}^{\nrom}
      \int_0^{2 \pi}
      \int_{-\infty}^{\infty}
      \frac{
        c_{i \ell}\herm
        \myparen*{
          H(\imag \omega, e^{\imag \omega_{\ppp}})
          - \hH(\imag \omega, e^{\imag \omega_{\ppp}})
        }
        b_{i \ell}
      }{
        \myparen*{-\imag \omega - \overline{\lambda_i}}
        \myparen*{e^{-\imag \omega_{\ppp}} - \overline{\pi_{\ell}}}^2
      }
      \dif{\omega}
      \dif{\omega_{\ppp}} \\
    & =
      \frac{1}{4 \pi^2 \imag^2}
      \sum_{i = 1}^{\nrom}
      \frac{1}{\overline{\pi_{\ell}}^2}
      \oint_{\partial\DD}
      \oint_{\imag \RR}
      \frac{
        \ppp
        c_{i \ell}\herm
        \myparen*{H(s, \ppp) - \hH(s, \ppp)}
        b_{i \ell}
      }{
        \myparen*{s - \myparen*{-\overline{\lambda_i}}}
        \myparen*{\ppp - \frac{1}{\overline{\pi_{\ell}}}}^2
      }
      \dif{s}
      \dif{\ppp} \\
    & =
      -\sum_{i = 1}^{\nrom}
      \frac{1}{\overline{\pi_{\ell}}^3}
      c_{i \ell}\herm
      \myparen*{
        \frac{\partial H}{\partial \ppp}\myparen*{-\overline{\lambda_i},
          \frac{1}{\overline{\pi_{\ell}}}}
        - \frac{\partial \hH}{\partial \ppp}\myparen*{-\overline{\lambda_i},
          \frac{1}{\overline{\pi_{\ell}}}}
      }
      b_{i \ell},
  \end{align*}
  which gives the final condition~\eqref{eq:grimm-mimo-cond-4}.
\end{proof}

Therefore, using \Cref{cor:cond} we are not only able recover the optimality
conditions~\eqref{eq:grimm-cond} for single-input single-output systems,
but also generalize them to systems with multiple inputs and outputs.
One can see parallels to the \ac{lti} case considered in \Cref{sec:lti-h2}.
The first two conditions~\eqref{eq:grimm-mimo-cond-1}
and~\eqref{eq:grimm-mimo-cond-2} are analogous to the left- and right-tangential
(Lagrange) interpolation conditions of~\eqref{eq:h2-cond-c}
and~\eqref{eq:h2-cond-b}, respectively.
Furthermore,~\eqref{eq:grimm-mimo-cond-3} and~\eqref{eq:grimm-mimo-cond-4}
resemble the bitangential Hermite conditions in~\eqref{eq:h2-cond-a}.
However, since we have a multivariate function $H$ in the $\HtwoLtwo$ case,
ordinary derivatives with respect to $s$ are now replaced by partial derivatives
respect to $s$ and $\ppp$.
And a bigger distinction is that in the Hermite
conditions~\eqref{eq:grimm-mimo-cond-3} and~\eqref{eq:grimm-mimo-cond-4},
interpolated function is not simply a derivative of $H$.
Rather a weighted sum of partial derivatives are interpolated.
Optimal interpolation points still result from \emph{mirroring} of the poles.
While the mirroring of $s$-poles is done with respect to the imaginary axis,
the mirroring of $\ppp$-poles is with respect to the unit circle.
This is not surprising since we used $\pset = \imag \RR \times \partial\DD$.
Therefore, the mirroring of $s$-poles resembles the continuous-time $\Htwo$
conditions~\eqref{eq:h2-cond} and
the mirroring of $\ppp$-poles resembles the discrete-time $\htwo$
conditions~\eqref{eq:h2-cond-dt}.

\section{Linear Time-invariant Systems: Discrete Measure}%
\label{sec:lti-ls}
In \Cref{sec:lti-h2}, we focused on $\Ltwo$-optimal modeling of \ac{lti}
systems using a continuous $\Ltwo$ measure in the frequency domain leading to
various systems-theoretic norms such as the $\Htwo$ norm.
In this section, we change our focus to a discrete measure and investigate the
resulting discrete \ac{ls} problem.

\subsection{Necessary Conditions for Discrete LS Problem}
Let $H$ be the transfer function of a continuous-time \ac{lti} system
(e.g., as in~\eqref{eq:lti-tf}).
Assume we only have access to the samples of $H$ at the sampling frequencies
$\{\imag \omega_i\}_{i = 1}^N$ where $\omega_i \in \RR$.
Let $H_i = H(\imag \omega_i) \in \CCoi$, for $i = 1, 2, \ldots, N$,
denote the corresponding frequency response data.
We assume that the sampling frequencies are closed under conjugation;
i.e., if $\imag \omega_k$ is a sampling point, then so is $-\imag \omega_k$.
In most cases, as in~\eqref{eq:lti-tf}, $H$ has a real state-space realization,
thus leading to $H(-\imag \omega_k) = \overline{H(\imag \omega_k)}$.
Therefore, the frequency response data $\{(\imag \omega_i, H_i)\}_{i = 1}^N$ is
closed under conjugation.

Given the sampling data $\{(\imag \omega_i, H_i)\}_{i = 1}^N$, the goal is to
find \iac{rom}~\eqref{eq:lti-cont-rom} with transfer function
$\hH(s) = \hC (s \hE - \hA)^{-1} \hB$ that minimizes the \ac{ls} error
\begin{align}\label{eq:lti-ls-obj}
  \obj(\hH)
  = \sum_{i = 1}^N \rho_i \normF*{H_i - \hH(\imag \omega_i)}^2,
\end{align}
where $\rho_i > 0$ are the weights
(equal for complex conjugate pairs of sampling frequencies).
We note that the \ac{ls} error in~\eqref{eq:lti-ls-obj} is a special case of the
$\Ltwo$ error~\eqref{eq:l2-cost} with the choices of
$\pset = \{\imag \omega_i\}_{i = 1}^N$,
$\yf(\imag \omega_i) = H_i$,
$\yr = \hH$, and
$\measure = \sum_{i = 1}^N \rho_i \delta_{\imag \omega_i}$,
where $\delta_{\imag \omega_i}$ is the Dirac measure at $\imag \omega_i$.
Thus, the rational \ac{ls} minimization problem~\eqref{eq:lti-ls-obj} directly
fits under our $\Ltwo$-optimal reduced-order modeling framework.

In~\cite{MliG22}, we have already considered this problem, i.e.,
the problem of approximating \ac{lti} systems from their frequency-domain data
using the discrete \ac{ls} measure.
We have devised a gradient-based optimization algorithm to minimize the
 \ac{ls} cost~\eqref{eq:lti-ls-obj}.
Our goal here is not algorithmic.
Here, using \Cref{cor:cond}, we derive new interpolatory necessary
conditions for \ac{ls} reduced-order modeling,
the first such conditions to the best of our knowledge, for the rational \ac{ls}
minimization problem.

Rational \ac{ls} fitting problem, i.e.,
minimizing the \ac{ls} cost using a rational function,
is an important and widely studied problem and there are various approaches to
tackling it, see,
e.g.,~\cite{GusS99,DrmGB15,DrmGB15a,HokM20,NakST18,BerG17},
and the references therein.
What we show here is that regardless of the underlying numerical algorithm,
a solution of the nonlinear rational \ac{ls} minimization problem is
interpolatory and satisfies specific bitangential Hermite interpolation
conditions.

\begin{theorem}
  Given the sampling data $\{(\imag \omega_i, H_i)\}_{i = 1}^N$,
  let the \ac{strom} $\hH(s) = \hC (s \hE - \hA)^{-1} \hB$ having the
  pole-residue form
  $\hH(s) = \sum_{j = 1}^{\nrom} \frac{c_j b_j\herm}{s - \lambda_j}$
  with pairwise distinct poles
  be a local minimum of $\obj$~\eqref{eq:lti-ls-obj}.
  Then
  \begin{subequations}\label{eq:lti-ls-cond}
    \begin{align}
      \label{eq:lti-ls-cond-1}
      \sum_{i = 1}^N
      \rho_i
      \frac{H_i b_k}{-\imag \omega_i - \overline{\lambda_k}}
      & =
        \sum_{i = 1}^N
        \rho_i
        \frac{\hH(\imag \omega_i) b_k}{
        -\imag \omega_i - \overline{\lambda_k}}, \\
      \label{eq:lti-ls-cond-2}
      \sum_{i = 1}^N
      \rho_i
      \frac{c_k\herm H_i}{-\imag \omega_i - \overline{\lambda_k}}
      & =
        \sum_{i = 1}^N
        \rho_i
        \frac{c_k\herm \hH(\imag \omega_i)}{
        -\imag \omega_i - \overline{\lambda_k}}, \\
      \label{eq:lti-ls-cond-3}
      \sum_{i = 1}^N
      \rho_i
      \frac{c_k\herm H_i b_k}{
      \myparen*{-\imag \omega_i - \overline{\lambda_k}}^2}
      & =
        \sum_{i = 1}^N
        \rho_i
        \frac{c_k\herm \hH(\imag \omega_i) b_k}{
        \myparen*{-\imag \omega_i - \overline{\lambda_k}}^2},
    \end{align}
  \end{subequations}
  for $k = 1, 2, \ldots, \nrom$.
\end{theorem}
\begin{proof}
  Let $\hH(s) = \hC (s \hE - \hA)^{-1} \hB$ and
  $\hT, \hS \in \CCrr$ be invertible matrices,
  with $c_j = \hC \hS^{-1} e_j$ and $b_j = \hB\tran \hT\mherm e_j$,
  as in \Cref{sec:lti-h2}, yielding the pole-residue form of $\hH$.
  Then, using
  $\pset = \{\imag \omega_i\}_{i = 1}^N$,
  $\yf(\imag \omega_i) = H_i$,
  $\yr = \hH$,
  $\xr = \hX$, and
  $\measure = \sum_{i = 1}^N \rho_i \delta_{\imag \omega_i}$
  in the left-hand side term in the $\Ltwo$-optimality
  condition~\eqref{eq:cond-C} gives
  \begin{align*}
    \int_{\pset} \yf(\pp) \xr(\pp)\herm \difm{\pp} \, \hT\mherm e_k
    & =
      \sum_{i = 1}^N
      \rho_i
      H_i
      \hB\tran \myparen*{\imag \omega_i \hE - \hA}\mherm
      \hT\mherm e_k
      =
      \sum_{i = 1}^N
      \rho_i
      \frac{H_i b_k}{-\imag \omega_i - \overline{\lambda_k}}.
  \end{align*}
  Similarly, using the left-hand side term in~\eqref{eq:cond-B}, we obtain
  \begin{align*}
    e_k\tran \hS^{-1} \int_{\pset} \xr_d(\pp) \yf(\pp) \difm{\pp}
    & =
      \sum_{i = 1}^N
      \rho_i
      e_k\tran \hS^{-1}
      \myparen*{\imag \omega_i \hE - \hA}\mherm \hC\tran
      H_i
      =
      \sum_{i = 1}^N
      \rho_i
      \frac{c_k\herm H_i}{-\imag \omega_i - \overline{\lambda_k}}.
  \end{align*}
  Lastly, the left-hand side term in~\eqref{eq:cond-A} corresponding to $\hA$
  shows
  \begin{align*}
    &
      {-e_k\tran} \hS^{-1}
      \int_{\pset} \car_2(\pp) \xr_d(\pp) \yf(\pp) \xr(\pp)\herm \difm{\pp}
      \, \hT\mherm e_k \\*
    & =
      \sum_{i = 1}^N
      \rho_i
      e_k\tran \hS^{-1}
      \myparen*{\imag \omega_i \hE - \hA}\mherm \hC\tran
      H_i
      \hB\tran \myparen*{\imag \omega_i \hE - \hA}\mherm
      \hT\mherm e_k
      =
      \sum_{i = 1}^N
      \rho_i
      \frac{c_k\herm H_i b_k}{
      \myparen*{-\imag \omega_i - \overline{\lambda_k}}^2}.
  \end{align*}
  Analogous calculations for the right-hand sides directly gives the
  conditions~\eqref{eq:lti-ls-cond}.
\end{proof}

We can rewrite the conditions~\eqref{eq:lti-ls-cond} to give a more immediate
interpolatory interpretation.
\begin{corollary}\label{cor:lti-ls-G}
  Given the sampling data $\{(\imag \omega_i, H_i)\}_{i = 1}^N$,
  let $\hH(s) = \sum_{j = 1}^{\nrom} \frac{c_j b_j\herm}{s - \lambda_j}$ have
  pairwise distinct poles and
  be a local minimizer of the \ac{ls} error~\eqref{eq:lti-ls-obj}.
  Furthermore, define the transfer functions
  \begin{equation}\label{eq:GandGhat}
    G(s) = \sum_{i = 1}^N \rho_i \frac{H_i}{s - \imag \omega_i}
    \quad \textnormal{and} \quad
    \hG(s) = \sum_{i = 1}^N \rho_i
    \frac{\hH(\imag \omega_i)}{s - \imag \omega_i}.
  \end{equation}
  Then
  \begin{subequations}\label{eq:ls-interp}
    \begin{align}
      \label{eq:ls-interp-1}
      G\myparen*{-\overline{\lambda_k}} b_k
      & =
        \hG\myparen*{-\overline{\lambda_k}} b_k, \\
      \label{eq:ls-interp-2}
      c_k\herm G\myparen*{-\overline{\lambda_k}}
      & =
        c_k\herm \hG\myparen*{-\overline{\lambda_k}}, \\
      \label{eq:ls-interp-3}
      c_k\herm G'\myparen*{-\overline{\lambda_k}} b_k
      & =
        c_k\herm \hG'\myparen*{-\overline{\lambda_k}} b_k,
    \end{align}
  \end{subequations}
  for $k = 1, 2, \ldots, \nrom$.
\end{corollary}
\begin{proof}
  The conditions~\eqref{eq:ls-interp-1} and~\eqref{eq:ls-interp-2} follow
  from~\eqref{eq:lti-ls-cond-1} and~\eqref{eq:lti-ls-cond-2},
  respectively,
  based on the definitions of $G$ and $\hG$ in~\eqref{eq:GandGhat}.
  The condition~\eqref{eq:lti-ls-cond-3} yields~\eqref{eq:ls-interp-3}
  after observing
  $G'(s) = -\sum_{i = 1}^N \rho_i \frac{H_i}{{(s - \imag \omega_i)}^2}$
  and
  $\hG'(s) = -\sum_{i = 1}^N \rho_i
  \frac{\hH(\imag \omega_i)}{{(s - \imag \omega_i)}^2}$.
\end{proof}
The conditions~\eqref{eq:ls-interp} illustrate that bitangential Hermite
interpolation is the necessary condition for the discrete $\Ltwo$ cost function
as well;
the interpolatory $\Ltwo$-optimal modeling framework equally applies.
The optimal approximant is still a bitangential Hermite interpolant,
but what is interpolated is different.
Here, two order-$N$ rational functions $G$ and $\hG$ interpolate each other
where $G$ depends on the evaluation of $H$ and
$\hG$ on the evaluations of $\hH$.
Yet, the interpolation points and directions are still determined by the poles
and residues of the optimal rational approximant $\hH$.
Mirror images of the reduced-order poles still appear as the interpolation
points.

\begin{remark}
  Alternatively, we can view the optimality conditions~\eqref{eq:lti-ls-cond} as
  discretized $\Htwo$-optimality conditions~\eqref{eq:h2-cond}.
  In particular, note that the interpolatory conditions~\eqref{eq:h2-cond} are
  equivalent to (using the Cauchy integral formula)
  \begin{align*}
    \int_{-\infty}^{\infty}
    \frac{H(\imag \omega) b_k}{-\imag \omega - \overline{\lambda_k}}
    \dif{\omega}
    & =
      \int_{-\infty}^{\infty}
      \frac{\hH(\imag \omega) b_k}{-\imag \omega - \overline{\lambda_k}}
      \dif{\omega}, \\
    \int_{-\infty}^{\infty}
    \frac{c_k\herm H(\imag \omega)}{-\imag \omega - \overline{\lambda_k}}
    \dif{\omega}
    & =
      \int_{-\infty}^{\infty}
      \frac{c_k\herm \hH(\imag \omega)}{-\imag \omega - \overline{\lambda_k}}
      \dif{\omega}, \\
    \int_{-\infty}^{\infty}
    \frac{c_i\herm H(\imag \omega) b_k}{
      \myparen*{-\imag \omega - \overline{\lambda_k}}^2}
    \dif{\omega}
    & =
      \int_{-\infty}^{\infty}
      \frac{c_i\herm \hH(\imag \omega) b_k}{
        \myparen*{-\imag \omega - \overline{\lambda_k}}^2
      }
      \dif{\omega},
  \end{align*}
  for $k = 1, 2, \ldots, \nrom$.
  Approximating these integrals (representing the $\Htwo$-optimality conditions)
  using a numerical quadrature with nodes $\omega_i$ and weights $\rho_i$ leads
  to the new optimality conditions~\eqref{eq:lti-ls-cond} for the cost
  function~\eqref{eq:lti-ls-obj}.
\end{remark}

\subsection{Numerical Example}%
\label{sec:penzl-ex}
To demonstrate the new interpolatory conditions from \Cref{cor:lti-ls-G} for the
rational discrete \ac{ls} minimization problem,
we use the Penzl's FOM model from the Niconet benchmark
collection~\cite{ChaV02}.
The model is \iac{lti} system of order $\nfom = 1006$ with $\nin = 1$ input and
$\nout = 1$ output.
We chose $\nrom = 2$ as the reduced order to make the illustration clear.
Clearly a high-fidelity approximation requires a higher order;
but our goal is here to illustrate the theory of \Cref{cor:lti-ls-G}.
For the data,
we take $50$ logarithmically-spaced frequencies between $10^0$ and $10^4$ on the
imaginary axis,
including the endpoints. With the inclusion of complex conjugate points,
we obtain $N = 100$ data points in the \ac{ls} problem~\eqref{eq:lti-ls-obj}.
We use the gradient-based optimization algorithm $\Ltwo$-Opt-PSF
in~\cite{MliG22}
(initialized with the \ac{rom} from iterative rational Krylov
algorithm~\cite{GugAB08})
and obtain \iac{rom} of order $\nrom = 2$ with poles at
$\lambda_1 \approx -4.7984$ and $\lambda_2 \approx -431.00$.

\Cref{fig:penzl-interp} shows the transfer functions $G$ and $\hG$,
defined in \Cref{cor:lti-ls-G},
when evaluated over the positive real axis.
Even though in the left plot,
we see overlap around $-\overline{\lambda_1}$ and $-\overline{\lambda_2}$,
it is not clear whether Hermite interpolation is achieved.
(Since $H$ is a single-input/single-output \ac{lti} system,
tangential interpolation boils downs to scalar interpolation, i.e.,
\eqref{eq:ls-interp-1} and \eqref{eq:ls-interp-2} coincide.)
To make the illustration clearer,
the right plot shows the difference between $G$ and $\hG$, and
here we see (from the shape and curvature of the error plot)
that $\hG$ is indeed a Hermite interpolant of $G$ at $-\overline{\lambda_1}$ and
$-\overline{\lambda_2}$ as the theory predicts.
We note that there is another interpolation point around $s \approx 15$.
However the derivative is not matched at this point;
thus no Hermite interpolation occurs at this point.

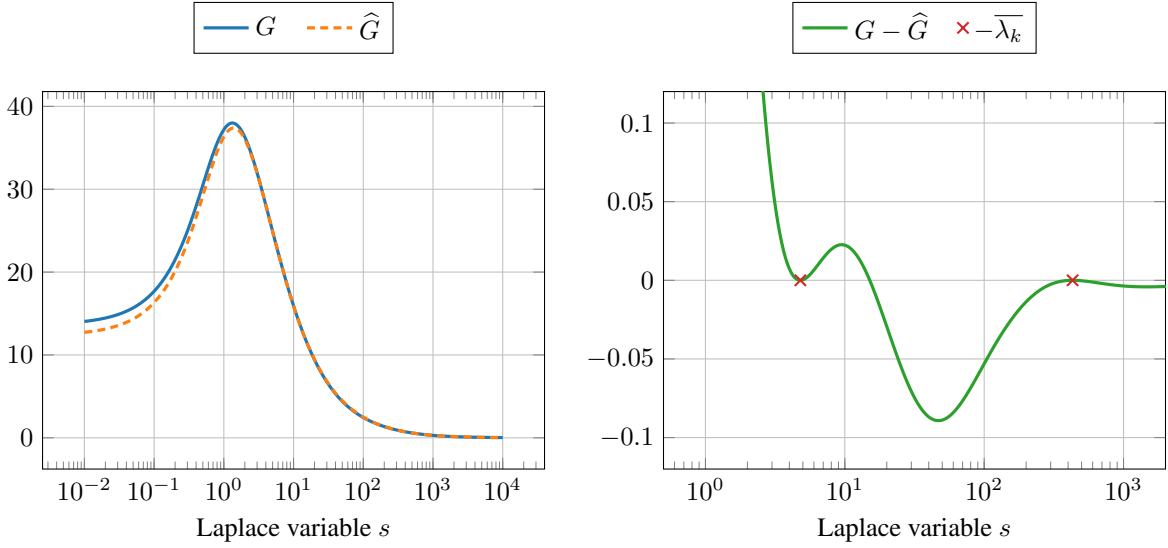
\begin{figure}[tb]
  \centering
  \tikzsetnextfilename{penzl-interp}
  \begin{tikzpicture}
    \begin{axis}[
        width=0.5\linewidth,
        height=0.4\linewidth,
        xlabel={Laplace variable $s$},
        xmode=log,
        grid=major,
        legend entries={$G$, $\hG$},
        legend columns=-1,
        legend style={
          at={(0.5, 1.1)},
          anchor=south,
          /tikz/every even column/.append style={column sep=2ex},
        },
        cycle list name=mpl,
      ]
      \addplot table [x=s, y=G] {fig/penzl_G.txt};
      \addplot table [x=s, y=Gr] {fig/penzl_G.txt};
    \end{axis}
    \begin{axis}[
        at={(0.5\linewidth, 0)},
        width=0.5\linewidth,
        height=0.4\linewidth,
        xmin=0.5,
        xmax=2000,
        ymin=-0.12,
        ymax=0.12,
        xlabel={Laplace variable $s$},
        xmode=log,
        yticklabel style={/pgf/number format/fixed},
        grid=major,
        legend entries={$G - \hG$, $-\overline{\lambda_k}$},
        legend columns=-1,
        legend style={
          at={(0.5, 1.1)},
          anchor=south,
          /tikz/every even column/.append style={column sep=2ex},
        },
      ]
      \addplot[mplC2, very thick] table [x=s, y=dG] {fig/penzl_G.txt};
      \addplot[only marks, mark=x, mark size=3, mplC3, thick]
        coordinates {(4.7984, 0) (431, 0)};
    \end{axis}
  \end{tikzpicture}
  \caption{Left: Transfer functions corresponding to the \ac{fom} and \ac{rom}
    data for the Penzl example ($\nfom = 1006$, $\nrom = 2$).
    Right: The difference between the modified outputs and
    the mirrored poles of the \ac{rom}.}%
  \label{fig:penzl-interp}
\end{figure}

The Python code used to compute the presented results can be obtained
from~\cite{Mli22b}.

\section{Stationary Parametric Problems}%
\label{sec:stat}
In \Cref{sec:lti,sec:lti-ls},
we have focused on approximating \ac{lti} systems.
Now we turn our attention to \emph{stationary} models and prove that
bitangential Hermite interpolation forms the necessary conditions for
$\Ltwo$-optimal \ac{mor} in this case as well.
Even though what is interpolated and the optimal interpolation points differ
from the \ac{lti} system case,
the optimality still requires bitangential Hermite interpolation.

\subsection{Necessary Conditions for Stationary Models}
Let us consider the stationary \ac{fom}
(resulting from, e.g., discretization of a stationary parametric partial
differential equation)
\begin{subequations}\label{eq:fom-simple}
  \begin{align}
    (\cAf_1 + \pp \cAf_2) \xf(\pp) & = \cBf, \\*
    \yf(\pp) & = \cCf \xf(\pp),
  \end{align}
\end{subequations}
where
$\xf(\pp) \in \RRf$ is the state,
$\yf(\pp) \in \RRoi$ is the output,
$\cAf_1, \cAf_2 \in \RRff$,
$\cBf \in \RRfi$,
$\cCf \in \RRof$, and
$\pset = [a, b] \subset \RR$ for some $a < b$.
Next, let
\begin{subequations}\label{eq:rom-simple}
  \begin{align}
    \myparen*{\cAr_1 + \pp \cAr_2} \xr(\pp) & = \cBr, \\*
    \yr(\pp) & = \cCr \xr(\pp),
  \end{align}
\end{subequations}
be \iac{strom} of order $\nrom$;
in particular,
$\xr(\pp) \in \RRr$ is the reduced state,
$\yr(\pp) \in \RRoi$ is the approximate output,
$\cAr_1, \cAr_2 \in \RRrr$,
$\cBr \in \RRri$, and
$\cCr \in \RRor$.
Our goal in this section is to show that the $\Ltwo$-optimal
\ac{strom}~\eqref{eq:rom-simple} satisfies special interpolation conditions.
More specially, we show that
a modified reduced-order output $\yr$ interpolates
a modified full-order output $\yf$
at special parameter points $\lambda_k$.

Note that both $\yf$ and $\yr$ are rational functions of the parameter $\pp$.
Motivated by the interpolation-based optimality conditions
from \Cref{sec:lti-h2} and those for frequency-limited $\Htwo$-optimal
\ac{mor}~\cite{VuiPA14},
we develop the interpolatory conditions for the $\Ltwo$-optimal approximation of
\acp{fom} of the form~\eqref{eq:fom-simple} using the pole-residue forms of
$\yf$ and $\yr$.

The pole-residue form of $\yr$~\eqref{eq:rom-simple} can be obtained similarly
as done for \ac{lti} system.
Let $\cAr_2$ be invertible and
$\cAr_2^{-1} \cAr_1$ have $\nrom$ distinct real eigenvalues.
Then, proceeding as in \Cref{sec:lti-h2},
let $\hT, \hS \in \RRrr$ be invertible matrices such that
$\hS\tran \cAr_1 \hT
= -\Lambda
= \mydiag{-\lambda_1, -\lambda_2, \ldots, -\lambda_{\nrom}}$ and
$\hS\tran \cAr_2 \hT = I$.
Then we obtain that
$\yr(\pp) = \sum_{j = 1}^{\nrom} \frac{c_j b_j\tran}{\pp - \lambda_j}$,
where
$c_j = \cCr \hT e_j \in \RRo$, and
$b_j = \cBr\tran \hS\tran e_j \in \RRi$
for $j = 1, 2, \ldots, \nrom$.

Inspired by the structure of \iac{fom} in the numerical example in
\Cref{sec:poisson},
we allow a slightly more general pole-residue form for $\yf$, namely
$\yf(\pp) = \Phi_0 + \sum_{i = 1}^{\nfom} \frac{\Phi_i}{\pp - \nu_i}$,
where $\Phi_0$ is a constant term,
$\nu_i$ are the poles and $\Phi_i$ the corresponding residues for
$i = 1, 2, \ldots, \nfom$.
The constant term $\Phi_0$ results from allowing
$\cAf_2$ to be a singular matrix
(as in the case of the numerical example in \Cref{sec:poisson}).
The details of the derivation of this pole-residue form are given in
\Cref{sec:dae}.

\begin{theorem}\label{thm:stat-cond}
  Let $\yr(\pp) = \sum_{j = 1}^{\nrom} \frac{c_j b_j\tran}{\pp - \lambda_j}$
  be the output of the \ac{strom}~\eqref{eq:rom-simple}
  with pairwise distinct $\lambda_j \in \RR \setminus [a, b]$.
  Furthermore,
  let $\yf(\pp) = \Phi_0 + \sum_{i = 1}^{\nfom} \frac{\Phi_i}{\pp - \nu_i}$
  be the output of the \ac{fom}~\eqref{eq:fom-simple},
  also with pairwise distinct $\nu_i \in \RR \setminus [a, b]$.
  For any $\sigma \in \RR \setminus \{a, b\}$,
  define the function $\fundef{f_\sigma}{\RR \setminus \{a, b\}}{\RR}$ as
  \begin{align*}
    f_{\sigma}(\pp)
    & =
      \begin{cases}
        \myparen*{
          \ln\abs*{\frac{\pp - b}{\pp - a}}
          - \ln\abs*{\frac{\sigma - b}{\sigma - a}}
        }
        \frac{1}{\pp - \sigma},
        & \text{if } \pp \neq \sigma, \\[1ex]
        \frac{b - a}{(\sigma - a) (\sigma - b)},
        & \text{if } \pp = \sigma.
      \end{cases}
  \end{align*}
  Furthermore, define the modified output functions
  $\fundef{Y, \hY}{\RR \setminus \{a, b\}}{\RR}$ as
  \begin{equation}\label{eq:stat-out-mod}
    Y(\pp)
    = \ln\abs*{\frac{\pp - b}{\pp - a}} \Phi_0
    + \sum_{i = 1}^{\nfom} f_{\nu_i}(\pp) \Phi_i
    \quad \text{and} \quad
    \hY(\pp) = \sum_{j = 1}^{\nrom} f_{\lambda_j}(\pp) c_j b_j\tran.
  \end{equation}
  Let~$\yr$ be an $\Ltwo$-optimal structured approximation of~$\yf$.
  Then,
  \begin{subequations}\label{eq:stat-cond}
    \begin{align}
      \label{eq:stat-cond-b}
      Y(\lambda_k) b_k & = \hY(\lambda_k) b_k, \\*
      \label{eq:stat-cond-c}
      c_k\tran Y(\lambda_k) & = c_k\tran \hY(\lambda_k), \\*
      \label{eq:stat-cond-bc}
      c_k\tran Y'(\lambda_k) b_k & = c_k\tran \hY'(\lambda_k) b_k,
    \end{align}
  \end{subequations}
  for $k = 1, 2, \ldots, \nrom$.
\end{theorem}
\begin{remark}\label{rem:thm-stat}
  Some remarks are in order before we prove \Cref{thm:stat-cond}.
  The optimality conditions~\eqref{eq:stat-cond} show that bitangential Hermite
  interpolation forms the necessary conditions for $\Ltwo$-optimal approximation
  over an interval $[a, b]$;
  thus extending the theory from \ac{lti} systems to stationary problems.
  However,
  what is to be interpolated is no longer the original function $\yf$ itself,
  instead the modified output $Y$ in~\eqref{eq:stat-out-mod} needs to be
  interpolated.
  Another major difference here is that
  the interpolation occurs at the reduced system poles,
  as opposed to at \emph{the mirror images of the poles}
  in the \ac{lti} system case.
  (Both modified outputs $Y$ and $\hY$ are well defined at the reduced poles as
  shown in the proof below.)
\end{remark}
\begin{proof}
  First note that $f_\sigma$ is continuously differentiable and
  its derivative is
  \begin{align*}
    f_{\sigma}'(\pp)
    & =
      \begin{cases}
        \myparen*{
        \frac{(b - a) (\pp - \sigma)}{(\pp - a) (\pp - b)}
        - \ln\abs*{\frac{\pp - b}{\pp - a}}
        + \ln\abs*{\frac{\sigma - b}{\sigma - a}}
        }
        \frac{1}{{(\pp - \sigma)}^2},
        & \text{if } \pp \neq \sigma, \\[1ex]
        \frac{(b - a) (a + b - 2 \sigma)}{2 {(\sigma - a)}^2 {(\sigma - b)}^2},
        & \text{if } \pp = \sigma.
      \end{cases}
  \end{align*}
  Since $f_\sigma$ is continuously differentiable, so are $Y$ and $\hY$.

  Next, compare the simple \ac{strom}~\eqref{eq:rom-simple} to the general
  case~\eqref{eq:rom-param-sep-form} to observe that we have
  $\qAr = 2$ with $\car_1(\pp) = 1$ and $\car_2(\pp) = \pp$,
  $\qBr = 1$ with $\cbr_1(\pp) = 1$, and
  $\qCr = 1$ with $\ccr_1(\pp) = 1$.
  It follows from the left-hand side of~\eqref{eq:cond-C} in \Cref{cor:cond}
  that
  \begin{align*}
    \int_a^b \yf(\pp) \xr(\pp)\tran \dif{\pp}
    & =
      \int_a^b
      \myparen*{\Phi_0 + \sum_{i = 1}^{\nfom} \frac{\Phi_i}{\pp - \nu_i}}
      \cBr\tran \myparen*{\cAr_1 + \pp \cAr_2}\mtran
      \dif{\pp}.
  \end{align*}
  Then, for $k = 1, 2, \ldots, \nrom$, we have
  \begin{align*}
    \int_a^b \yf(\pp) \xr(\pp)\tran \dif{\pp} \, \hT\mtran e_k
    & =
      \int_a^b
      \myparen*{\Phi_0 + \sum_{i = 1}^{\nfom} \frac{\Phi_i}{\pp - \nu_i}}
      \cBr\tran \hS \myparen*{-\Lambda + \pp I}\mtran e_k
      \dif{\pp} \\
    & =
      \int_a^b
      \frac{\Phi_0 b_k}{\pp - \lambda_k}
      \dif{\pp}
      +
      \sum_{i = 1}^{\nfom}
      \int_a^b \frac{\Phi_i b_k}{(\pp - \nu_i) (\pp - \lambda_k)}
      \dif{\pp}.
  \end{align*}
  If $\nu_i = \lambda_k$, then
  \begin{align*}
    \int_a^b
    \frac{\Phi_i b_k}{(\pp - \nu_i) (\pp - \lambda_k)}
    \dif{\pp}
    & =
      \int_a^b \frac{\Phi_i b_k}{{(\pp - \nu_i)}^2}
      \dif{\pp}
      =
      \frac{(b - a) \Phi_i b_k}{(\nu_i - a) (\nu_i - b)}
      =
      f_{\nu_i}(\lambda_k) \Phi_i b_k.
  \end{align*}
  Otherwise,
  \begin{align*}
    &
      \int_a^b
      \frac{\Phi_i b_k}{(\pp - \nu_i) (\pp - \lambda_k)}
      \dif{\pp}
      =
      \frac{1}{\lambda_k - \nu_i}
      \int_a^b
      \myparen*{\frac{1}{\pp - \lambda_k} - \frac{1}{\pp - \nu_i}}
      \dif{\pp}
      \,\Phi_i b_k \\
    & =
      \myparen*{\ln\abs*{\frac{\lambda_k - b}{\lambda_k - a}}
        - \ln\abs*{\frac{\nu_i - b}{\nu_i - a}}}
      \frac{\Phi_i}{\lambda_k - \nu_i} b_k
      =
      f_{\nu_i}(\lambda_k) \Phi_i b_k.
  \end{align*}
  Therefore, we obtain
  \begin{align*}
    \int_a^b \yf(\pp) \xr(\pp)\tran \dif{\pp} \, \hT\mtran e_k
    & =
      \ln\abs*{\frac{\lambda_k - b}{\lambda_k - a}} \Phi_0 b_k
      + \sum_{i = 1}^{\nfom}
      f_{\nu_i}(\lambda_k) \Phi_i b_k
      =
      Y(\lambda_k) b_k.
  \end{align*}
  From the right-hand side of~\eqref{eq:cond-C}, we similarly find
  \begin{align*}
    &
      \int_a^b \yr(\pp) \xr(\pp)\tran \dif{\pp} \, \hT\mtran e_k \\*
    & =
      \sum_{\substack{j = 1 \\ j \neq k}}^{\nrom}
      \myparen*{\ln\abs*{\frac{\lambda_k - b}{\lambda_k - a}}
        - \ln\abs*{\frac{\lambda_j - b}{\lambda_j - a}}}
      \frac{c_j b_j\tran}{\lambda_k - \lambda_j} b_k
      + \frac{(b - a) c_k b_k\tran}{(\lambda_k - a) (\lambda_k - b)} b_k
      =
      \hY(\lambda_k) b_k.
  \end{align*}
  Therefore, we get the right-tangential interpolation
  conditions~\eqref{eq:stat-cond-b}.

  Using~\eqref{eq:cond-B}, we similarly find
  \begin{equation*}
    e_k\tran \hS^{-1} \int_a^b \xr_d(\pp) \yf(\pp) \dif{\pp}
    = c_k\tran Y(\lambda_k)
    \quad \textnormal{and} \quad
    e_k\tran \hS^{-1} \int_a^b \xr_d(\pp) \yr(\pp) \dif{\pp}
    = c_k\tran \hY(\lambda_k),
  \end{equation*}
  which yield the left-tangential interpolation
  conditions~\eqref{eq:stat-cond-c}.

  Using~\eqref{eq:cond-A} corresponding to $\cAr_1$, we find
  \begin{align*}
    e_k\tran \hS^{-1}
    \int_a^b \xr_d(\pp) \yf(\pp) \xr(\pp)\tran \dif{\pp}
    \, \hT\mtran e_k
    & =
      \int_a^b \frac{c_k\tran \Phi_0 b_k}{{(\pp - \lambda_k)}^2} \dif{\pp} \\*
    & \quad
      + \sum_{i = 1}^{\nfom}
      \int_a^b \frac{c_k\tran \Phi_i b_k}{(\pp - \nu_i) {(\pp - \lambda_k)}^2}
      \dif{\pp}.
  \end{align*}
  If $\nu_i = \lambda_k$, then
  \begin{align*}
    \int_a^b
    \frac{c_k\tran \Phi_i b_k}{(\pp - \nu_i) {(\pp - \lambda_k)}^2}
    \dif{\pp}
    & =
      \int_a^b \frac{c_k\tran \Phi_i b_k}{{(\pp - \nu_i)}^3}
      \dif{\pp}
      =
      \frac{(b - a) (a + b - 2 \nu_i) c_k\tran \Phi_i b_k}%
      {2 {(\nu_i - a)}^2 {(\nu_i - b)}^2} \\
    & =
      f_{\nu_i}'(\lambda_k) c_k\tran \Phi_i b_k.
  \end{align*}
  Otherwise,
  \begin{align*}
    \int_a^b
    \frac{c_k\tran \Phi_i b_k}{(\pp - \nu_i) {(\pp - \lambda_k)}^2}
    \dif{\pp}
    & =
      \frac{1}{{(\lambda_k - \nu_i)}^2}
      \int_a^b
      \myparen*{
        \frac{1}{\pp - \nu_i}
        - \frac{1}{\pp - \lambda_k}
        + \frac{\lambda_k - \nu_i}{{(\pp - \lambda_k)}^2}
      }
      \dif{\pp}
      \, c_k\tran \Phi_i b_k \\
    & =
      \myparen*{
        \ln\abs*{\frac{\nu_i - b}{\nu_i - a}}
        - \ln\abs*{\frac{\lambda_k - b}{\lambda_k - a}}
        + \frac{(b - a) (\lambda_k - \nu_i)}{(\lambda_k - a) (\lambda_k - b)}
      }
      \frac{c_k\tran \Phi_i b_k}{{(\lambda_k - \nu_i)}^2} \\
    & =
      f_{\nu_i}'(\lambda_k) c_k\tran \Phi_i b_k.
  \end{align*}
  Therefore,
  \begin{align*}
    &
      e_k\tran \hS^{-1}
      \int_a^b \xr_d(\pp) \yf(\pp) \xr(\pp)\tran \dif{\pp}
      \, \hT\mtran e_k \\*
    & =
      \myparen*{\frac{1}{\lambda_k - b} - \frac{1}{\lambda_k - a}}
      c_k\tran \Phi_0 b_k
      + \sum_{i = 1}^{\nfom}
      f_{\nu_i}'(\lambda_k) c_k\tran \Phi_i b_k
      =
      c_k\tran Y'(\lambda_k) b_k.
  \end{align*}
  Similarly, we find that
  \begin{align*}
    e_k\tran \hS^{-1}
    \int_a^b \xr_d(\pp) \yr(\pp) \xr(\pp)\tran \dif{\pp}
    \, \hT\mtran e_k
    & =
      c_k\tran \hY'(\lambda_k) b_k,
  \end{align*}
  which proves the final bitangential Hermite interpolation
  conditions~\eqref{eq:stat-cond-bc}.
\end{proof}
\Cref{thm:stat-cond} has shown that bitangential Hermite interpolation,
which is at the core of $\Htwo$-optimal approximation of \ac{lti} systems,
also naturally appears in the approximation of parametric stationary problems.
In the \ac{lti} system setting,
these interpolatory conditions have been at the core of many algorithmic
developments and extended to various different settings;
see, e.g.,~\cite{GugAB06,GugAB08,BeaG09,AntBG10,BauBBetal11,GugPBS12,GugSW13,%
  FlaBG13,AniBGA13,BeaB14,FlaG15,BreBG15,CarGB18,BenGW21a,BenGW21b},
and the references therein.
Similar potential extensions and algorithmic developments for the stationary
parametric case will be a topic of future research.

\subsection{Numerical Example}%
\label{sec:poisson}
As we did for \Cref{cor:lti-ls-G} in \Cref{sec:penzl-ex},
we demonstrate the new interpolatory results of \Cref{thm:stat-cond}
using a numerical example.
Following~\cite{MliG22},
we consider the Poisson equation over the unit square $\Omega = {(0, 1)}^2$ with
homogeneous Dirichlet boundary conditions:
\begin{align*}
  -\nabla \cdot (d(z, \pp) \nabla \xf(z, \pp)) & = 1,
  & z \in \Omega, \\*
  \xf(z, \pp) & = 0,
  & z \in \partial \Omega,
\end{align*}
where
$d(z, \pp) = z_1 + \pp (1 - z_1)$ and
$\pset = [0.1, 10]$.
After a finite element discretization,
we obtain the \ac{fom} of the form~\eqref{eq:fom-simple}
with $\nfom = 1089$ and $\nin = 1$,
and the choice of $\cCf = \cBf\tran$ (and $\nout = 1$).
We use the gradient-based optimization algorithm $\Ltwo$-Opt-PSF
in~\cite{MliG22}
(initialized with the \ac{strom} resulting from applying a reduced-basis
approach~\cite{BenGQetal20b}) and
obtain \iac{strom} of order $\nrom = 2$ with poles at
$\lambda_1 \approx -3.2777$ and $\lambda_2 \approx -0.30509$.

To numerically verify the interpolatory conditions,
we need the pole-residue forms of $\yf$ and $\yr$ as stated in
\Cref{thm:stat-cond}.
Note that these computations are not needed by the $\Ltwo$-Opt-PSF algorithm and
are done here just to illustrate the interpolation theory.
Since $\cAr_2$ is invertible,
the pole-residue form of $\yr$ directly follows as explained above,
right before \Cref{thm:stat-cond}.
The situation is more involved for $\yf$ since $\cAf_2$ in~\eqref{eq:fom-simple}
is rank-deficient; it has numerical rank $\nfom_2 = 961$.
We use the procedure explained in \Cref{sec:dae} to compute the
pole-residue form of $\yf$.
With the pole-residue forms of $\yf$ and $\yr$,
we can now numerically evaluate the modified outputs $Y$ and $\hY$ defined
in~\eqref{eq:stat-out-mod} and illustrate the interpolation result.

The left plot in \Cref{fig:poisson-interp} shows that
$Y$ and $\hY$ almost overlap,
making it unclear whether Hermite interpolation is achieved.
To illustrate the results better,
the right plot in \Cref{fig:poisson-interp} depicts the difference $Y - \hY$
around the location of the poles of the \ac{strom},
clearly demonstrating (based on the shape and curvature of the error plot)
that the \ac{strom} satisfies the Hermite interpolation
conditions~\eqref{eq:stat-cond} of \Cref{thm:stat-cond}.
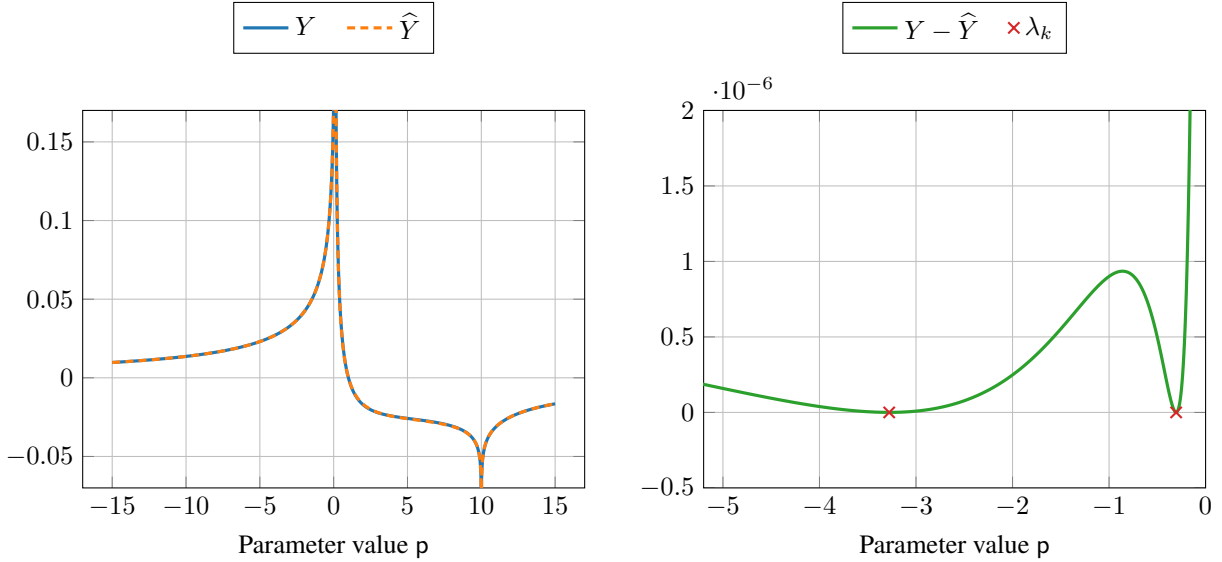
\begin{figure}[tb]
  \centering
  \tikzsetnextfilename{poisson-interp}
  \begin{tikzpicture}
    \begin{axis}[
        width=0.5\linewidth,
        height=0.4\linewidth,
        xmin=-17,
        xmax=17,
        ymin=-0.07,
        ymax=0.17,
        xlabel={Parameter value $\pp$},
        yticklabel style={/pgf/number format/fixed},
        grid=major,
        legend entries={$Y$, $\hY$},
        legend columns=-1,
        legend style={
          at={(0.5, 1.15)},
          anchor=south,
          /tikz/every even column/.append style={column sep=2ex},
        },
        cycle list name=mpl,
      ]
      \addplot table [x=p, y=Y] {fig/poisson_Y_Yr.txt};
      \addplot table [x=p, y=Yr] {fig/poisson_Y_Yr.txt};
    \end{axis}
    \begin{axis}[
        at={(0.5\linewidth, 0)},
        width=0.5\linewidth,
        height=0.4\linewidth,
        xmin=-5.2,
        xmax=0,
        ymin=-5e-7,
        ymax=2e-6,
        xlabel={Parameter value $\pp$},
        grid=major,
        legend entries={$Y - \hY$, $\lambda_k$},
        legend columns=-1,
        legend style={
          at={(0.5, 1.15)},
          anchor=south,
          /tikz/every even column/.append style={column sep=2ex},
        },
      ]
      \addplot[mplC2, very thick] table [x=p, y=dY] {fig/poisson_dY.txt};
      \addplot[only marks, mark=x, mark size=3, mplC3, thick]
        coordinates {(-3.27772, 0) (-0.30509, 0)};
    \end{axis}
  \end{tikzpicture}
  \caption{Left: Modified outputs of the \ac{fom} and \ac{rom} for the Poisson
    example ($\nfom = 1089$, $\nrom = 2$).
    Right: Difference between the modified outputs and
    the poles of the \ac{rom}.}%
  \label{fig:poisson-interp}
\end{figure}

It is interesting to note that the interpolation points $\lambda_1$ and
$\lambda_2$ are outside the parameter space $\pset = [a, b]$.
This is in agreement with the $\Htwo$-optimality conditions~\eqref{eq:h2-cond},
where interpolation is enforced away from the imaginary axis, in particular, in
the open right half-plane.
In contrast to the reduced-basis methods that choose the greedy sampling points
in the parameter interval of interest,
$\Ltwo$-optimal reduced-order modeling necessitates interpolation of a modified
output $Y$ outside the domain of interest.

The Python code used to compute the presented results can be obtained
from~\cite{Mli22b}.

\section{Conclusion}%
\label{sec:conclusion}
We developed a unifying framework for $\Ltwo$-optimal
interpolatory reduced-order modeling.
In particular, we showed that the $\Ltwo$-optimality conditions resulting from this framework naturally cover
known interpolatory conditions for $\Htwo$-optimal \ac{mor} of \ac{lti} systems,
both for continuous-time and discrete-time cases.
Furthermore, they lead to interpolatory conditions for $\HtwoLtwo$-optimal
\ac{mor} of multi-input/multi-output parametric \ac{lti} systems.
We also derived novel bitangential Hermite interpolation conditions
for rational \ac{ls} problems and
for a class of stationary parametric problems.
These results illustrate that bitangential Hermite interpolation appears as the main
tool for $\Ltwo$-optimality across different domains.

Interpolatory conditions for $\Htwo$-optimal \ac{mor} lead to various numerical
algorithms for \ac{mor},
such as the iterative rational Krylov algorithm.
Algorithmic implications of the new interpolatory conditions for the discrete
\ac{ls} measure and
the stationary problems are interesting avenues to investigate.

\appendix

\section{Proof of \texorpdfstring{\Cref{lem:kron-grad}}{Lemma}}%
\label{sec:lemma-proof}
We have that for any $\Delta A \in \RR^{n \times n}$,
\begin{align*}
  G(A + \Delta A)
  & =
    F((A + \Delta A) \otimes B)
    =
    F(A \otimes B + \Delta A \otimes B) \\*
  & =
    G(A)
    + \ipF*{\nabla F(A \otimes B)}{\Delta A \otimes B}
    + o(\normF{\Delta A}).
\end{align*}
Let $\fundef{P_A}{\RR^{n m \times n m}}{\RR^{n \times n}}$ be defined as
\begin{equation*}
  P_A(Y) = \argmin_{A_1 \in \RR^{n \times n}} \normF{Y - A_1 \otimes B}^2.
\end{equation*}
Using the orthogonality property of the least squares approximation,
it follows that $\ipF{Y - P_A(Y) \otimes B}{A_1 \otimes B} = 0$ for all
$A_1 \in \RR^{n \times n}$.
Then,
\begin{align*}
  &
    \ipF*{\nabla F(A \otimes B)}{\Delta A \otimes B}
    =
    \ipF*{P_A(\nabla F(A \otimes B)) \otimes B}{\Delta A \otimes B} \\
  & =
    \trace*{(P_A(\nabla F(A \otimes B)) \otimes B)\tran (\Delta A \otimes B)}
    =
    \trace*{P_A(\nabla F(A \otimes B))\tran \Delta A \otimes B\tran B} \\
  & =
    \trace*{P_A(\nabla F(A \otimes B))\tran \Delta A} \trace*{B\tran B}
    =
    \ipF*{P_A(\nabla F(A \otimes B))}{\Delta A} \normF{B}^2 \\
  & =
    \ipF*{\normF{B}^2 P_A(\nabla F(A \otimes B))}{\Delta A}.
\end{align*}
Therefore, it follows that $G$ is differentiable at $A$ and
\begin{equation}\label{eq:grad-G}
  \nabla G(A)
  =
  \normF{B}^2
  P_A(\nabla F(A \otimes B)).
\end{equation}
To find $P_A(Y)$,
define $V_A = \{A_1 \otimes B : A_1 \in \RR^{n \times n}\}$.
Note that $V_A$ is a subspace of $\RR^{n m \times n m}$.
We see that
\(
  \mybrace{
    \frac{1}{\normF{B}} e_k e_{\ell}\tran \otimes B :
    k, \ell \in \{1, 2, \ldots, n\}
  }
\)
is an orthonormal basis for $V_A$.
Therefore
\begin{align*}
  P_{A}(Y) \otimes B
  & =
    \sum_{k, \ell = 1}^n
    \ipF*{Y}{\frac{1}{\normF{B}} e_k e_{\ell}\tran \otimes B}
    \myparen*{\frac{1}{\normF{B}} e_k e_{\ell}\tran \otimes B} \\
  & =
    \frac{1}{\normF{B}^2}
    \sum_{k, \ell = 1}^n
    \trace*{\myparen*{e_k e_{\ell}\tran \otimes B}\tran Y}
    \myparen*{e_k e_{\ell}\tran \otimes B}.
\end{align*}
Note that
\begin{align*}
  &
    \trace*{\myparen*{e_k e_{\ell}\tran \otimes B}\tran Y}
    =
    \trace*{\myparen*{e_{\ell} e_k\tran \otimes B\tran} Y}
    =
    \trace*{
      \myparen*{e_{\ell} \otimes B_R\herm}
      \myparen*{e_k\tran \otimes B_L\herm}
      Y
    } \\
  & =
    \trace*{
      \myparen*{e_k\tran \otimes B_L\herm}
      Y
      \myparen*{e_{\ell} \otimes B_R\herm}
    }
    =
    \sum_{j = 1}^m
    e_j\tran
    \myparen*{e_k\tran \otimes B_L\herm}
    Y
    \myparen*{e_{\ell} \otimes B_R\herm}
    e_j \\
  & =
    \sum_{j = 1}^m
    \myparen*{e_k\tran \otimes e_j\tran B_L\herm}
    Y
    \myparen*{e_{\ell} \otimes B_R\herm e_j}
    =
    \sum_{j = 1}^m
    e_k\tran
    \myparen*{I_n \otimes e_j\tran B_L\herm}
    Y
    \myparen*{I_n \otimes B_R\herm e_j}
    e_{\ell} \\
  & =
    \mybrack*{
      \sum_{j = 1}^m
      \myparen*{I_n \otimes e_j\tran B_L\herm}
      Y
      \myparen*{I_n \otimes B_R\herm e_j}
    }_{k \ell}.
\end{align*}
Now,
\begin{align*}
  P_A(Y) \otimes B
  & =
    \frac{1}{\normF{B}^2}
    \myparen*{
      \sum_{k, \ell = 1}^n
      \mybrack*{
        \sum_{j = 1}^m
        \myparen*{I_n \otimes e_j\tran B_L\herm}
        Y
        \myparen*{I_n \otimes B_R\herm e_j}
      }_{k \ell}
      e_k e_{\ell}\tran
    }
    \otimes B \\
  & =
    \myparen*{
      \frac{1}{\normF{B}^2}
      \sum_{j = 1}^m
      \myparen*{I_n \otimes e_j\tran B_L\herm}
      Y
      \myparen*{I_n \otimes B_R\herm e_j}
    }
    \otimes B.
\end{align*}
Using that $B$ is nonzero, we obtain
\begin{align*}
  P_A(Y)
  & =
    \frac{1}{\normF{B}^2}
    \sum_{j = 1}^m
    \myparen*{I_n \otimes e_j\tran B_L\herm}
    Y
    \myparen*{I_n \otimes B_R\herm e_j},
\end{align*}
and the expression for $\nabla G(A)$ follows using~\eqref{eq:grad-G}.
The expression for $\nabla H(B)$ can be found analogously.

\section{Pole-residue Form with a Constant Term}%
\label{sec:dae}
To convert the form in~\eqref{eq:fom-simple} to the pole-residue form
in \Cref{thm:stat-cond},
write $\cAf_2 = U V\tran$ for some $U, V \in \RR^{\nfom \times \nfom_2}$ of full
column rank.
Then, using the Sherman-Morrison-Woodbury formula~\cite{GolVL13},
we obtain
\begin{align*}
  \yf(\pp)
  & = \cCf \myparen*{\cAf_1 + \pp U V\tran}^{-1} \cBf \\
  & = \cCf
    \myparen*{
      \cAf_1^{-1}
      - \pp \cAf_1^{-1} U
        \myparen*{I_{\nfom_2} + \pp V\tran \cAf_1^{-1} U}^{-1}
        V\tran \cAf_1^{-1}
    }
    \cBf \\
  & = \cCf \cAf_1^{-1} \cBf
    - \pp
    \cCf \cAf_1^{-1} U
    \myparen*{I_{\nfom_2} + \pp V\tran \cAf_1^{-1} U}^{-1}
    V\tran \cAf_1^{-1} \cBf.
\end{align*}
Next, let $T \in \CC^{\nfom_2 \times \nfom_2}$ be an invertible matrix such that
$V\tran \cAf_1^{-1} U = T D T^{-1}$ for
$D = \mydiag{d_1, d_2, \ldots, d_{\nfom_2}}$.
Furthermore, define $\cCf_U = \cCf \cAf_1^{-1} U$ and
$\cBf_V = V\tran \cAf_1^{-1} \cBf$.
Continuing the above derivation, we obtain
\begin{align*}
  \yf(\pp)
  & =
    \cCf \cAf_1^{-1} \cBf
    - \pp
      \cCf_U
      \myparen*{I_{\nfom_2} + \pp T D T^{-1}}^{-1}
      \cBf_V \\
  & =
    \cCf \cAf_1^{-1} \cBf
    - \pp
      \cCf_U T
      \myparen*{I_{\nfom_2} + \pp D}^{-1}
      T^{-1} \cBf_V \\
  & =
    \cCf \cAf_1^{-1} \cBf
    - \sum_{i = 1}^{\nfom_2} \frac{\pp
        \cCf_U T e_i
        e_i\tran T^{-1} \cBf_V
      }{1 + \pp d_i}.
\end{align*}
Using
$\frac{\pp}{1 + \pp d_i}
= \frac{1}{d_i} - \frac{\frac{1}{d_i^2}}{\pp + \frac{1}{d_i}}$,
$\sum_{i = 1}^{\nfom_2} \frac{1}{d_i} e_i e_i\tran = D^{-1}$, and
$T D^{-1} T^{-1} = \myparen{V\tran \cAf_1^{-1} U}^{-1}$
yields
\begin{align*}
  \yf(\pp)
  & =
    \cCf \cAf_1^{-1} \cBf
    - \cCf_U \myparen*{V\tran \cAf_1^{-1} U}^{-1} \cBf_V
    + \sum_{i = 1}^{\nfom_2}
      \frac{\frac{1}{d_i^2}
        \cCf_U T e_i
        e_i\tran T^{-1} \cBf_V
      }{\pp + \frac{1}{d_i}}.
\end{align*}
Therefore, in the pole-residue form \Cref{thm:stat-cond}, we have
\begin{align*}
  \Phi_0
  & =
    \cCf \cAf_1^{-1} \cBf
    - \cCf \cAf_1^{-1} U
      \myparen*{V\tran \cAf_1^{-1} U}^{-1}
      V\tran \cAf_1^{-1} \cBf, \\
  \Phi_i
  & =
    \cCf \cAf_1^{-1} U T e_i
    e_i\tran T^{-1} V\tran \cAf_1^{-1} \cBf
    / d_i^2, \\
  \nu_i
  & =
    -1 / d_i,
\end{align*}
for $i = 1, 2, \ldots, \nfom_2$.




\begin{thebibliography}{BGKVW10}

\bibitem[ABG10]{AntBG10}
A.~C. Antoulas, C.~A. Beattie, and S.~Gugercin.
\newblock Interpolatory model reduction of large-scale dynamical systems.
\newblock In J.~Mohammadpour and K.~M. Grigoriadis, editors, {\em Efficient
  Modeling and Control of Large-Scale Systems}, pages 3--58. Springer US,
  Boston, MA, 2010.
\newblock \href {https://doi.org/10.1007/978-1-4419-5757-3_1}
  {\path{doi:10.1007/978-1-4419-5757-3_1}}.

\bibitem[ABG20]{AntBG20}
A.~C. Antoulas, C.~A. Beattie, and S.~G{\"u}{\u{g}}ercin.
\newblock {\em Interpolatory methods for model reduction}.
\newblock Computational Science and Engineering 21. SIAM, Philadelphia, PA,
  2020.
\newblock \href {https://doi.org/10.1137/1.9781611976083}
  {\path{doi:10.1137/1.9781611976083}}.

\bibitem[ABGA13]{AniBGA13}
B.~Ani{\'{c}}, C.~Beattie, S.~Gugercin, and A.~C. Antoulas.
\newblock Interpolatory weighted-{$\mathcal{H}_2$} model reduction.
\newblock {\em Automatica}, 49(5):1275--1280, 2013.
\newblock \href {https://doi.org/10.1016/j.automatica.2013.01.040}
  {\path{doi:10.1016/j.automatica.2013.01.040}}.

\bibitem[BB12]{BenB12}
P.~Benner and T.~Breiten.
\newblock Interpolation-based {$\mathcal{H}_2$}-model reduction of bilinear
  control systems.
\newblock {\em {SIAM} J. Matrix Anal. Appl.}, 33:859--885, 2012.
\newblock \href {https://doi.org/10.1137/110836742}
  {\path{doi:10.1137/110836742}}.

\bibitem[BB14]{BeaB14}
C.~A. Beattie and P.~Benner.
\newblock {$\mathcal{H}_2$}-optimality conditions for structured dynamical
  systems.
\newblock Preprint MPIMD/14-18, Max Planck Institute Magdeburg, 2014.
\newblock URL: \url{https://csc.mpi-magdeburg.mpg.de/preprints/2014/18/}.

\bibitem[BBBG11]{BauBBetal11}
U.~Baur, C.~A. Beattie, P.~Benner, and S.~Gugercin.
\newblock Interpolatory projection methods for parameterized model reduction.
\newblock {\em {SIAM} J. Sci. Comput.}, 33(5):2489--2518, 2011.
\newblock \href {https://doi.org/10.1137/090776925}
  {\path{doi:10.1137/090776925}}.

\bibitem[BBG15]{BreBG15}
T.~Breiten, C.~Beattie, and S.~Gugercin.
\newblock Near-optimal frequency-weighted interpolatory model reduction.
\newblock {\em Systems Control Lett.}, 78:8--18, 2015.
\newblock \href {https://doi.org/10.1016/j.sysconle.2015.01.005}
  {\path{doi:10.1016/j.sysconle.2015.01.005}}.

\bibitem[BG09]{BeaG09}
C.~Beattie and S.~Gugercin.
\newblock Interpolatory projection methods for structure-preserving model
  reduction.
\newblock {\em Systems Control Lett.}, 58(3):225--232, 2009.
\newblock \href {https://doi.org/10.1016/j.sysconle.2008.10.016}
  {\path{doi:10.1016/j.sysconle.2008.10.016}}.

\bibitem[BG17a]{BeaG17}
C.~Beattie and S.~Gugercin.
\newblock {\em Chapter 7: Model Reduction by Rational Interpolation}, pages
  297--334.
\newblock SIAM, 2017.
\newblock \href {https://doi.org/10.1137/1.9781611974829.ch7}
  {\path{doi:10.1137/1.9781611974829.ch7}}.

\bibitem[BG17b]{BerG17}
Mario Berljafa and Stefan G{\"{u}}ttel.
\newblock The {RKFIT} algorithm for nonlinear rational approximation.
\newblock {\em {SIAM} J. Sci. Comput.}, 39(5):A2049--A2071, 2017.
\newblock \href {https://doi.org/10.1137/15M1025426}
  {\path{doi:10.1137/15M1025426}}.

\bibitem[BGG18]{BenGG18}
P.~Benner, P.~Goyal, and S.~Gugercin.
\newblock {$\mathcal{H}_2$}-quasi-optimal model order reduction for
  quadratic-bilinear control systems.
\newblock {\em {SIAM} J. Matrix Anal. Appl.}, 39(2):983--1032, 2018.
\newblock \href {https://doi.org/10.1137/16M1098280}
  {\path{doi:10.1137/16M1098280}}.

\bibitem[BGKVW10]{BunKVetal10}
A.~Bunse-Gerstner, D.~Kubalinska, G.~Vossen, and D.~Wilczek.
\newblock {$h_2$}-norm optimal model reduction for large scale discrete
  dynamical {MIMO} systems.
\newblock {\em J. Comput. Appl. Math.}, 233(5):1202--1216, 2010.
\newblock \href {https://doi.org/10.1016/j.cam.2008.12.029}
  {\path{doi:10.1016/j.cam.2008.12.029}}.

\bibitem[BGTQ{\etalchar{+}}20]{BenGQetal20b}
P.~Benner, S.~Grivet-Talocia, A.~Quarteroni, G.~Rozza, W.~Schilders, and L.~M.
  Silveira, editors.
\newblock {\em Model Order Reduction: Volume 2: Snapshot-Based Methods and
  Algorithms}.
\newblock De Gruyter, Berlin, Boston, 2020.
\newblock \href {https://doi.org/10.1515/9783110671490}
  {\path{doi:10.1515/9783110671490}}.

\bibitem[BGW21a]{BenGW21b}
P.~Benner, S.~Gugercin, and S.~W.~R. Werner.
\newblock Structure-preserving interpolation for model reduction of parametric
  bilinear systems.
\newblock {\em Automatica}, 132(109799):1--9, 2021.
\newblock \href {https://doi.org/10.1016/j.automatica.2021.109799}
  {\path{doi:10.1016/j.automatica.2021.109799}}.

\bibitem[BGW21b]{BenGW21a}
P.~Benner, S.~Gugercin, and S.~W.~R. Werner.
\newblock Structure-preserving interpolation of bilinear control systems.
\newblock {\em Adv. Comput. Math.}, 47(43):1--38, May 2021.
\newblock \href {https://doi.org/10.1007/s10444-021-09863-w}
  {\path{doi:10.1007/s10444-021-09863-w}}.

\bibitem[BOCW17]{Benetal17}
P.~Benner, M.~Ohlberger, A.~Cohen, and K.~Willcox.
\newblock {\em Model Reduction and Approximation}.
\newblock SIAM, Philadelphia, PA, 2017.
\newblock \href {https://doi.org/10.1137/1.9781611974829}
  {\path{doi:10.1137/1.9781611974829}}.

\bibitem[CV02]{ChaV02}
Y.~Chahlaoui and P.~{Van Dooren}.
\newblock A collection of benchmark examples for model reduction of linear time
  invariant dynamical systems.
\newblock Technical Report 2002-2, SLICOT Working Note, 2002.
\newblock Available from \url{www.slicot.org}.

\bibitem[DGB15a]{DrmGB15}
Z.~Drma{\v{c}}, S.~Gugercin, and C.~Beattie.
\newblock Quadrature-based vector fitting for discretized {$\mathcal{H}_{2}$}
  approximation.
\newblock {\em {SIAM} J. Sci. Comput.}, 37(2):A625--A652, 2015.
\newblock \href {https://doi.org/10.1137/140961511}
  {\path{doi:10.1137/140961511}}.

\bibitem[DGB15b]{DrmGB15a}
Z.~Drma{\v{c}}, S.~Gugercin, and C.~Beattie.
\newblock Vector fitting for matrix-valued rational approximation.
\newblock {\em {SIAM} J. Sci. Comput.}, 37(5):A2346--A2379, 2015.
\newblock \href {https://doi.org/10.1137/15M1010774}
  {\path{doi:10.1137/15M1010774}}.

\bibitem[DS11]{DruS11}
V.~Druskin and V.~Simoncini.
\newblock Adaptive rational {K}rylov subspaces for large-scale dynamical
  systems.
\newblock {\em Systems Control Lett.}, 60(8):546--560, 2011.
\newblock \href {https://doi.org/10.1016/j.sysconle.2011.04.013}
  {\path{doi:10.1016/j.sysconle.2011.04.013}}.

\bibitem[DSZ14]{DruSZ14}
V.~Druskin, V.~Simoncini, and M.~Zaslavsky.
\newblock Adaptive tangential interpolation in rational {K}rylov subspaces for
  {MIMO} dynamical systems.
\newblock {\em {SIAM} J. Matrix Anal. Appl.}, 35(2):476--498, 2014.
\newblock \href {https://doi.org/10.1137/120898784}
  {\path{doi:10.1137/120898784}}.

\bibitem[FAB17]{FenAB17}
L.~Feng, A.~C. Antoulas, and P.~Benner.
\newblock Some a posteriori error bounds for reduced order modelling of
  (non-)parametrized linear systems.
\newblock {\em {ESAIM}: Math. Model. Numer. Anal.}, 51(6):2127--2158, 2017.
\newblock \href {https://doi.org/10.1051/m2an/2017014}
  {\path{doi:10.1051/m2an/2017014}}.

\bibitem[FB19]{FenB19}
L.~Feng and P.~Benner.
\newblock A new error estimator for reduced-order modeling of linear parametric
  systems.
\newblock {\em {IEEE} Trans. Microw. Theory Techn.}, 67(12):4848--4859, 2019.
\newblock \href {https://doi.org/10.1109/TMTT.2019.2948858}
  {\path{doi:10.1109/TMTT.2019.2948858}}.

\bibitem[FB21]{FenB21}
L.~Feng and P.~Benner.
\newblock On error estimation for reduced-order modeling of linear
  non-parametric and parametric systems.
\newblock {\em {ESAIM}: Math. Model. Numer. Anal.}, 55(2):561--594, 2021.
\newblock \href {https://doi.org/10.1051/m2an/2021001}
  {\path{doi:10.1051/m2an/2021001}}.

\bibitem[FBG13]{FlaBG13}
G.~Flagg, C.~A. Beattie, and S.~Gugercin.
\newblock Interpolatory {$H_{\infty}$} model reduction.
\newblock {\em Systems Control Lett.}, 62(7):567--574, 2013.
\newblock \href {https://doi.org/10.1016/j.sysconle.2013.03.006}
  {\path{doi:10.1016/j.sysconle.2013.03.006}}.

\bibitem[FG15]{FlaG15}
G.~Flagg and S.~Gugercin.
\newblock Multipoint {V}olterra series interpolation and {$\mathcal{H}_2$}
  optimal model reduction of bilinear systems.
\newblock {\em {SIAM} J. Matrix Anal. Appl.}, 36(2):549--579, 2015.
\newblock \href {https://doi.org/10.1137/130947830}
  {\path{doi:10.1137/130947830}}.

\bibitem[GAB06]{GugAB06}
S.~Gugercin, A.~C. Antoulas, and C.~A. Beattie.
\newblock A rational {K}rylov iteration for optimal {$\mathcal{H}_2$} model
  reduction.
\newblock In {\em Proc. of the 17th International Symposium on Mathematical
  Theory of Networks and Systems}, pages 1665--1667, 2006.

\bibitem[GAB08]{GugAB08}
S.~Gugercin, A.~C. Antoulas, and C.~Beattie.
\newblock {$\mathcal{H}_2$} model reduction for large-scale linear dynamical
  systems.
\newblock {\em {SIAM} J. Matrix Anal. Appl.}, 30(2):609--638, 2008.
\newblock \href {https://doi.org/10.1137/060666123}
  {\path{doi:10.1137/060666123}}.

\bibitem[GPBvdS12]{GugPBS12}
S.~Gugercin, R.~V. Polyuga, C.~Beattie, and A.~van~der Schaft.
\newblock Structure-preserving tangential interpolation for model reduction of
  port-{H}amiltonian systems.
\newblock {\em Automatica}, 48(9):1963--1974, 2012.
\newblock \href {https://doi.org/10.1016/j.automatica.2012.05.052}
  {\path{doi:10.1016/j.automatica.2012.05.052}}.

\bibitem[Gri18]{Gri18}
A.~R. Grimm.
\newblock {\em Parametric Dynamical Systems: Transient Analysis and Data Driven
  Modeling}.
\newblock PhD thesis, Virginia Polytechnic Institute and State University,
  2018.
\newblock URL: \url{http://hdl.handle.net/10919/83840}.

\bibitem[GS99]{GusS99}
B.~Gustavsen and A.~Semlyen.
\newblock Rational approximation of frequency domain responses by vector
  fitting.
\newblock {\em {IEEE} Trans. Power Del.}, 14(3):1052--1061, 1999.
\newblock \href {https://doi.org/10.1109/61.772353}
  {\path{doi:10.1109/61.772353}}.

\bibitem[GSW13]{GugSW13}
S.~Gugercin, T.~Stykel, and S.~Wyatt.
\newblock Model reduction of descriptor systems by interpolatory projection
  methods.
\newblock {\em {SIAM} J. Sci. Comput.}, 35(5):B1010--B1033, 2013.
\newblock \href {https://doi.org/10.1137/130906635}
  {\path{doi:10.1137/130906635}}.

\bibitem[GV13]{GolVL13}
G.~H. Golub and C.~F. {Van Loan}.
\newblock {\em Matrix Computations}.
\newblock The Johns Hopkins University Press, fourth edition, 2013.

\bibitem[HM20]{HokM20}
J.~M. Hokanson and C.~C. Magruder.
\newblock {$\mathcal{H}_2$}-optimal model reduction using projected nonlinear
  least squares.
\newblock {\em {SIAM} J. Sci. Comput.}, 42(6):A4017--A4045, 2020.
\newblock \href {https://doi.org/10.1137/19M1247863}
  {\path{doi:10.1137/19M1247863}}.

\bibitem[HRS16]{HesRS16}
J.~S. Hesthaven, G.~Rozza, and B.~Stamm.
\newblock {\em Certified reduced basis methods for parametrized partial
  differential equations}.
\newblock Springer Briefs in Mathematics. Springer, Switzerland, 2016.
\newblock \href {https://doi.org/10.1007/978-3-319-22470-1}
  {\path{doi:10.1007/978-3-319-22470-1}}.

\bibitem[MG22]{MliG22}
P.~Mlinari{\'c} and S.~Gugercin.
\newblock {$\mathcal{L}_2$}-optimal reduced-order modeling using
  parameter-separable forms.
\newblock arXiv preprint 2206.02929, 2022.
\newblock \href {https://doi.org/10.48550/arXiv.2206.02929}
  {\path{doi:10.48550/arXiv.2206.02929}}.

\bibitem[ML67]{MeiL67}
L.~Meier and D.~Luenberger.
\newblock Approximation of linear constant systems.
\newblock {\em {IEEE} Trans. Autom. Control}, 12(5):585--588, 1967.
\newblock \href {https://doi.org/10.1109/TAC.1967.1098680}
  {\path{doi:10.1109/TAC.1967.1098680}}.

\bibitem[Mli22]{Mli22b}
P.~Mlinari{\'c}.
\newblock {$\mathcal{L}_2$}-optimal interpolation experiments, August 2022.
\newblock URL: \url{https://github.com/pmli/l2-opt-interp-ex/tree/v1}.

\bibitem[NST18]{NakST18}
Y.~Nakatsukasa, O.~S{\`{e}}te, and L.~N. Trefethen.
\newblock The {AAA} algorithm for rational approximation.
\newblock {\em {SIAM} J. Sci. Comput.}, 40(3):A1494--A1522, 2018.
\newblock \href {https://doi.org/10.1137/16M1106122}
  {\path{doi:10.1137/16M1106122}}.

\bibitem[QMN16]{QuaMN16}
A.~Quarteroni, A.~Manzoni, and F.~Negri.
\newblock {\em Reduced basis methods for partial differential equations}.
\newblock UNITEXT. Springer Cham, Switzerland, 2016.
\newblock \href {https://doi.org/10.1007/978-3-319-15431-2}
  {\path{doi:10.1007/978-3-319-15431-2}}.

\bibitem[RGB18]{CarGB18}
A.~Carracedo Rodriguez, S.~Gugercin, and J.~Borggaard.
\newblock Interpolatory model reduction of parameterized bilinear dynamical
  systems.
\newblock {\em Adv. Comput. Math.}, 44(6):1887--1916, December 2018.
\newblock \href {https://doi.org/10.1007/s10444-018-9611-y}
  {\path{doi:10.1007/s10444-018-9611-y}}.

\bibitem[Son98]{Son98}
Eduardo~D. Sontag.
\newblock {\em Mathematical Control Theory}.
\newblock Springer New York, 1998.
\newblock \href {https://doi.org/10.1007/978-1-4612-0577-7}
  {\path{doi:10.1007/978-1-4612-0577-7}}.

\bibitem[VPVA14]{VuiPA14}
P.~Vuillemin, C.~Poussot-Vassal, and D.~Alazard.
\newblock Poles residues descent algorithm for optimal frequency-limited
  {$\mathcal{H}_2$} model approximation.
\newblock In {\em European Control Conference (ECC)}, pages 1080--1085, 2014.
\newblock \href {https://doi.org/10.1109/ECC.2014.6862152}
  {\path{doi:10.1109/ECC.2014.6862152}}.

\end{thebibliography}
\addcontentsline{toc}{chapter}{References}
\newcommand{\etalchar}[1]{$^{#1}$}

\end{document}